\documentclass[12pt]{article}
\usepackage{harvard}
\usepackage{bbm}
\usepackage{wrapfig}
\usepackage{proof}
\usepackage{amsbsy,eucal}  
\usepackage{amsmath}
\usepackage{amsfonts}
\usepackage{latexsym}
\usepackage[all]{xy}  
\DeclareSymbolFont{extraup}{U}{zavm}{m}{n}
\DeclareMathSymbol{\varheart}{\mathalpha}{extraup}{86}
\DeclareMathSymbol{\vardiamond}{\mathalpha}{extraup}{87}

\newcommand{\bbar}{\overline{b}}
\newcommand{\cbar}{\overline{c}}

\newcommand{\dbar}{\overline{d}}
\newcommand{\ebar}{\overline{e}}
\newcommand{\fbar}{\overline{f}}
\newcommand{\gbar}{\overline{g}}

\newcommand{\nott}{\neg}

\newcommand{\andd}{\wedge}
\newcommand{\iif}{\rightarrow}
\newcommand{\vv}{{\cal V}}

\newcommand{\abar}{\overline{a}}
\newcommand{\xbar}{\overline{x}}

\newcommand{\ybar}{\overline{y}}
\renewcommand{\hbar}{\overline{h}}
\newcommand{\set}[1]{\{ #1 \}}
\newcommand{\pair}[1]{\langle #1 \rangle}
\newcommand{\pairalt}[1]{\langle\!\langle #1 \rangle\!\rangle}

\newcommand{\Sdagleq}{\mathcal{L}(\bP)}

  \newcommand{\pbar}{\overline{p}}
  \newcommand{\existsset}{\exists}
    \newcommand{\Model}{{\cal M}}
\newcommand{\Nodel}{{\cal N}}
\renewcommand{\SS}{{\cal{S}}}
\newcommand{\Structure}{\SS}
\newcommand{\rem}[1]{\relax}
\newcommand{\quadiff}{\quad \mbox{iff}\quad}

\newcommand{\TT}{{\cal T}}
\newcommand{\proves}{\vdash}
\newcommand{\cardparentheses}[1]{\cardinal{(#1)}}
\newcommand{\card}[1]{\cardinal{#1}}
\newcommand{\cardinal}[1]{\mbox{card}\,#1}
\newcommand{\qbar}{\overline{q}}

\newcommand{\cS}{\protect{\mathcal{S}}}  %
\newcommand{\existsgeq}{\exists^{\geq}}
\newcommand{\existsmore}{\exists^{>}}
\newcommand{\more}{\exists^{>}}
\newcommand{\leqc}{\leq_c}
\newcommand{\equivc}{\equiv_c}
\newcommand{\bP}{{\mathbf{P}}}

\newcommand{\semantics}[1]{[\![ #1 ]\!]}

\newcommand{\leqmore}{\leqmorestrict}
\newcommand{\leqmorestrict}{<_{more}}

\newcommand{\phibar}{\overline{\phi}}

\renewcommand{\phi}{\varphi}

\newlength{\mathfrwidth}
\newsavebox{\mathfrbox}
\newenvironment{mathframe}
    {\begin{lrbox}{\mathfrbox}\begin{minipage}{\mathfrwidth}}
    {\end{minipage}\end{lrbox}\noindent\fbox{\usebox{\mathfrbox}}}

\usepackage{amsthm}
\theoremstyle{definition}
\newtheorem{theorem}{Theorem}[section]
\newtheorem{lemma}[theorem]{Lemma}
\newtheorem{claim}[theorem]{Claim}
\newtheorem{proposition}[theorem]{Proposition}
\newtheorem{corollary}[theorem]{Corollary}
\newtheorem{definition}[theorem]{Definition}

 \newtheorem{ex}[theorem]{Example}
     \newtheorem{remark}[theorem]{Remark}

\makeindex

\begin{document}
\title{Syllogistic Logic with Cardinality Comparisons,\\  On Infinite Sets}
\author{Lawrence S.~Moss\footnote{This work was partially supported by a grant from the Simons Foundation ($\#$245591 to Lawrence Moss).}
\ and\ Sel\c{c}uk Topal}
\providecommand{\keywords}[1]{\textbf{\textit{Keywords: }} #1}

\date{\today}
\maketitle

\begin{abstract}
    {This paper enlarges classical syllogistic logic with assertions having to do with 
        comparisons between the sizes of sets.
        So it concerns a logical system whose sentences
        are of the following forms:     
        {\sf All $x$ are $y$} and {\sf Some $x$ are $y$},  
        {\sf There are at least as many $x$ as $y$}, and {\sf There are more $x$ than $y$}.
        Here $x$ and $y$ range over subsets (not elements) of a given \emph{infinite} set.
         Moreover, $x$ and $y$ may appear complemented (i.e., as $\xbar$ and $\ybar$), with the natural meaning. 
          We formulate a logic for our language that is based on the classical syllogistic. The main result is a soundness/completeness
        theorem.  There are efficient algorithms for  proof search and model construction.      }
    
\end{abstract}

\keywords{Logic of natural languages; completeness; syllogistic logic; 
    
    \hspace{2.7cm}infinite sets; cardinality comparisons }

\section{Introduction}
\label{section-Introduction}

This paper is a contribution to the study of \emph{extended syllogistic logics}.  The idea is to take the classical syllogistic
$\cS$
 as a base system and to extend this base with reasoning 
power for some phenomenon which is not expressible in first-order logic, and then to show that on top of the small base, the resulting system is still well-behaved.
We thus take the logical system $\cS$ whose sentences are of the form   {\sf All $p$ are $q$} and {\sf Some $p$ are $q$},  and   {\sf No $p$ are $q$}.
We shall explain the semantics shortly, but for the time being we do without the semantics. 
To add expressive power, we also allow the variables to be complemented. 
So for each $p$ we would also have $\pbar$; the interpretation is the set theoretic complement of $\semantics{p}$. The complementation operation
also provides an extra natural sentence coming from the Aristotelian syllogistic to the system: {\sf Some $p$ are not $q$}.
In our work, this is expressed via the syntax {\sf Some $p$ are  $\qbar$}.
We call this system $\cS^\dag$.

In formulating $\cS^\dag$, we may omit  {\sf No $x$ are $y$}, as it is equivalent to
 {\sf All $x$ are $\ybar$}.     This paper enlarges the syntax with 
 assertions   {\sf There are at least as many $x$ as $y$}, and {\sf There are more $x$ than $y$}.   
In this paper, we  require the universe to be infinite.   The case of finite universes 
was treated in~\cite{moss:forDunn}
The syntax and semantics of the systems
        are exactly the same.
    But the main point of the paper is to craft a sound and complete proof system, and the proof system in this paper is different from the one in~\cite{moss:forDunn}.

\setcounter{equation}{0}

\label{section-syntax}

\paragraph{Syntax: preliminary} 
Here is the definition of 
the logical language $\Sdagleq$  which we study in this paper.
A \emph{noun system} is a pair
$$\begin{array}{lcl}
\bP & = & (P, \overline{\phantom{p}})
\end{array}
$$
where $P$ is a set, and $p\mapsto \pbar$ is an involutive operation on $P$.
Thus, $\overline{\pbar} = p$ for all $p\in \bP$.
We call the elements of $\bP$ \emph{nouns}.
All Roman letters $p$, $q$, $x$, $y$, $\ldots$ denote nouns.

We fix a noun system $\bP$ for the rest of this paper. 

\paragraph{Syntax of $\Sdagleq$:}
$\Sdagleq$
has sentences $\forall(p,q)$ and 
$\exists(p,q)$,  $\existsgeq(p,q)$, $\more(p,q)$.   We read ``$\existsgeq(p,q) $'' as ``there are at least as many $p$ as $q$'',
and we read ``$\more(p,q) $'' as ``there are more $p$ than $q$.''
There are no connectives.

\paragraph{Semantics}  

\begin{definition}
A \emph{structure} $\SS  $ is  a pair consisting of a set $S$ and an 
\emph{interpretation function} $\semantics{\ }: \bP \to \mathcal{P}(S)$, where $\mathcal{P}(S)$ is the power set of $S$.
That is, a structure is a set $S$ together with a function which interprets nouns as subsets of $S$.
A \emph{model} is a structure $\Model = (M, \semantics{\ })$ with the additional property that for all $p$,
$\semantics{\pbar} = M\setminus\semantics{p}$ for all nouns $p$.
\label{def-structure-model}
\end{definition}

We define the \emph{satisfaction} relation between models and sentences as follows:
\begin{equation}
\begin{array}{lcl}
\Model\models  \forall(p,q) & \quadiff & \semantics{p} \subseteq \semantics{q} \\
\Model\models  \exists(p,q) & \quadiff & \semantics{p} \cap \semantics{q}  \neq \emptyset \\
\Model\models  \existsgeq(p,q)  & \quadiff &   \card{\semantics{p}} \geq   \card{\semantics{q}}  \\
\Model\models  \more(p,q)  & \quadiff   &  \card{\semantics{p}} >   \card{\semantics{q}}   \\
\end{array}
\label{semantics}
\end{equation}
On the right in (\ref{semantics}), the symbol $\card{S}$ stands for the cardinality   of a
given set $S$.

\begin{remark}
    Our treatment of cardinality and of cardinal comparison is
    completely standard.  We recall that   $\cardparentheses{X} \leq \cardparentheses{Y}$ means that there is a one-to-one 
    $f: X\to Y$.  
    Also, $\cardparentheses{X} <  \cardparentheses{Y}$  means that   $\cardparentheses{X} \leq \cardparentheses{Y}$, but $\nott(\cardparentheses{Y} \leq \cardparentheses{X})$.
    All of our work with infinite cardinals is also standard.   We remind the reader that
    if $X$ and $Y$ are infinite sets, then $\cardparentheses{X\cup Y} = \max(\cardparentheses{X},\cardparentheses{Y})$. 
    
    In addition, we use the Axiom of Choice at several points in this paper. 
  For example, if $\Model\models  \more(p,q) $ is false, then $\Model\models  \existsgeq(q,p) $ is true. We also use it equivalent, Zorn's Lemma, in the proof of  Lemma~\ref{lemma-use-zorn}.
    
Let us  also mention that the main result in this paper  needs very little set theory.   The only notable point is that there are infinitely many infinite cardinals.
\end{remark}

\begin{remark}
    We study our logic on infinite models $\Model$ in this paper.  However, 
    $\semantics{p}$ is not required to be infinite.  It might even be empty.
\end{remark}

\rem{
In (\ref{semantics}), the letters $p$ and $q$ range over all nouns,
not just over the raw variables.     In other words,  we have saved a lot of needless repetition
by allowing $p$ and $q$ to be either raw variables or complemented variables.
We continue this practice throughout the paper.
}

Special attention should be given to sentences $\existsgeq(p,\pbar)$.
When the universe of a model is finite, 
$\existsgeq(p,\pbar)$ says that there are at least as many $p$s as non-$p$s.
This sentence $\existsgeq(p,\pbar)$   might, therefore, be read as 
``the $p$'s are at least half of the objects in the universe.''
Similarly, $\existsgeq(\pbar,p)$ might be read as ``the $p$'s are at most half of the objects in the universe.''
We can also read 
$\more(p,\pbar)$ as ``the $p$'s are more than half of the objects in the universe,''
and 
$\more(\pbar,p)$   as ``the $p$'s are less than half of the objects in the universe.''
However, in this paper, all models $\Model$ have an infinite  universe $M$.
So our   talk about ``half'' must be taken with a grain of salt.  
To say that 
$\Model\models\existsgeq(p,\pbar)$ in fact says that $\card{\semantics{p}} = \card{M}$.
To say that 
$\existsmore(p,\pbar)$   says that $\card{\semantics{p}} = \card{M}$ and also $\card{\semantics{\pbar}} < \card{M}$.

\begin{definition}
    \label{semanticnegation}
For every sentence $\phi$, there is a sentence $\phibar$
    such that $\Model\models\phibar$ iff $\Model\not\models\phi$.
    Here is how this works:
    $$\begin{array}{l@{\qquad}l}
    \mbox{$\phi$} & \phibar\\
    \hline
    \forall(p,q) & \exists(p,\qbar) \\
    \exists(p,q) & \forall(p,\qbar) \\
    \existsgeq(p,q) & \more(q,p) \\
    \more(p,q)  & \existsgeq(q,p)\\
    \end{array}
    $$
The operation symbol $\overline{\phantom{\phi}}$ on sentences is not part of the  language $\Sdagleq$.    
\end{definition}

\begin{definition} Let $\Gamma\cup\set{\phi}\subseteq \Sdagleq$.
$\Model\models\Gamma$ means that $\Model\models\phi$ for all $\psi\in\Gamma$ 
and 
$\Gamma\models\phi$ means that if $\Model\models\Gamma$,
then also 
$\Model\models\psi$.
\end{definition}

\subsection{Proof system}

The main goal of this paper is to present a sound and complete logic for the semantic consequence relation $\Gamma\models\phi$.
This section describes the system and provides examples.
The rules of the logic are shown in Figures~\ref{figure-finite-setting-soundpart} and~\ref{figure-secondpart}.

Let $\Gamma$ be a set of sentences in $\Sdagleq$.
A \emph{proof tree over $\Gamma$} 
is a
finite  tree  $\TT$ 
whose nodes are labelled with sentences,
and
each node is either a leaf node labelled with an element of $\Gamma$,
or else matches one of the rules in the proof system  in 
one of  Figures~\ref{figure-finite-setting-soundpart} or \ref{figure-secondpart}.
\emph{$\Gamma \proves \phi$} means that there is a  proof tree $\TT$
for  over $\Gamma$
whose root is labelled $\phi$.

\begin{figure}[p]
    \begin{mathframe}
        $$
        \begin{array}{c@{\quad}c@{\quad}c}
        \infer[{\mbox{\sc (axiom)}}]{\forall(p,p)}{}
        &
        \infer[{\mbox{\sc (barbara)}}]{\forall(n,q)}{\forall(n,p) & \forall(p,q)} 
        \\  \\
        \infer[{\mbox{\sc (some)}}]{\exists(p,p)}{ \exists(p,q)} 
        &
        \infer[{\mbox{\sc (conversion)}}]{\exists(p,q)}{ \exists(q,p)} 
        \\  \\  
        \infer[{\mbox{\sc (darii)}}]{\exists(p,q)}{ \exists(p,n)  & \forall(n,q)} 
        & 
        \infer[\mbox{\sc (anti)}]{\forall(\qbar,\pbar)}{ \forall(p,q)} 
        \\  \\
        \infer[{\mbox{\sc (zero)}}]{\forall(p,q)}{ \forall(p,\pbar)} 
        &
        \infer[{\mbox{\sc (one)}}]{\forall(q,p)}{ \forall(\pbar,p)}  
        \\  \\
        
        \infer[{\mbox{\sc (subset-size)}}]{\existsgeq(q,p)}{ \forall(p,q)} 
        & 
        \infer[{\mbox{\sc (card-trans)}}]{\existsgeq(n,q)}{\existsgeq(n,p) &\existsgeq(p,q)}
        \\  \\
        \infer[{\mbox{\sc (card-$\exists$)}}]{\exists(q,q)}{\exists(p,p) & \existsgeq(q,p)}
        &
        \infer[{\mbox{\sc (more-at least)}}]{\existsgeq(p,q)}{\more(p,q) }
        \\  \\   
        \infer[{\mbox{\sc (more-left)}}]{\more(n,q)}{\more(n,p) & \existsgeq(p,q) }
        & 
        \infer[{\mbox{\sc (more-right)}}]{\more(n,q)}{\existsgeq(n,p) & \more(p,q) }  
        \\  \\
        \infer[{\mbox{\sc (more-some)}}]{\exists(p,\qbar)}{\more(p,q) }
        &
        \\   \\
        & 
        
        \\ \\
        \infer[{\mbox{\sc (x)}}]{\phi}{\exists(p,q) & \forall(q,\qbar) } 
        &
        \infer[{\mbox{\sc (x-card)}}]{\phi}{\more(p,q) & \existsgeq(q,p) } 
        \end{array}  
        $$
    \end{mathframe}
    \caption{The first part of the rules of our logical system.}   
    \label{figure-finite-setting-soundpart}
\end{figure}

\rem{
\begin{proposition}\label{prop-is-finite}
    Let $\Model$ be any model,  let $x$, $y$, and $z$ be any nouns, 
    and assume that $\existsmore(p,x)$ and $\existsmore(q,\xbar)$ are true in $\Model$.
    Then $\Model$ is finite.
\end{proposition}

\begin{proof}
    Suppose not; let $\kappa = \cardparentheses{M}$.  So $\kappa$  is infinite.
    Let $\lambda = \card{\semantics{x}}$, and let  $\lambda' = \card{\semantics{\xbar}}$
    On the one hand, $\lambda,\lambda'  < \kappa$.  But 
    $$\begin{array}{lclclclclclcl}
    \kappa & = &\card{M} & = & \cardparentheses{\semantics{x}\cup\semantics{\xbar}}
    &\leq & \max(\lambda, \lambda') &< & \kappa.
    \end{array}
    $$
    This is a contradiction.
\end{proof}
}

\rem{
At this point, we have some rules which are sound for both finite and infinite models
(Figure~\ref{figure-finite-setting-soundpart}).   And we know some rules which are not sound for infinite models,
(Figure~\ref{figure-finite-setting-unsoundpart}).   These rules must be dropped in this paper.
Doing this gives a sound system, but the result would not be complete.
Thus, we search for some extra rules which are sound for infinite models.
One way to do this is to weaken hypotheses in the unsound rules.  Another way to find rules is to 
prove the completeness theorem carefully and to ``get stuck'', thereby noticing when a new rule is needed.
In fact, we used both strategies in finding the rules of our system.
The new rules are listed in Figure~\ref{figure-secondpart}.
}

\begin{proposition} [Soundness]
    If $\Gamma\proves \phi$, then $\Gamma\models \phi$.
    \label{prop-soundness}
\end{proposition}

\begin{proof} 
    By induction on the heights of proof trees.
    The proof reduces to showing that all of the rules are individually sound.    
The rules in Figure~\ref{figure-finite-setting-soundpart} discussed  in~\cite{moss:forDunn}, and in any case  they are easy to justify.
 
 We show the soundness of  the rules in 
Figure~\ref{figure-secondpart}   by considering
a single (arbitrary) infinite model $\Model$, say of size $\kappa$.
We assume the hypotheses of the rules in turn and show the conclusions.
This will show that our logic is sound.

Let us show the soundness of ({\sc non-empty}).
If  $\existsgeq(p,\pbar)$, then $\semantics{p} =  \kappa > 0$.

The rule ({\sc non-empty-more}) is sound because infinite sets are not empty.

Consider ({\sc weak-more-anti}).
If $\card{\semantics{q}} > \card{\semantics{p}}$,  then $\card{\semantics{p}} < \kappa$.
And so $\card{\semantics{\pbar}} = \kappa$.   Thus 
$\card{\semantics{\pbar}} \geq \card{\semantics{x}}$ for all $x$.

\rem{
    Consider  ({\sc soft-strict-half}). 
    The hypotheses imply that 
    $\card{\semantics{p}} = \kappa$, $\card{\semantics{\pbar}} < \kappa$, 
    and  $\card{\semantics{q}} = \kappa$.
    If $\forall(p,\qbar)$ held in $\Model$,
    then $\forall(q,\pbar)$ would also be true, by the soundness of ({\sc anti}).   
    So $\card{\semantics{q}} \leq  \card{\semantics{\pbar}}$, by soundness of ({\sc subset-size}).
    But 
    $\card{\semantics{\pbar}} < \kappa = \card{\semantics{q}}$, and so we have a contradiction.
}

Turning to ({\sc up}), suppose that 
$\semantics{x}$ is at least as large as both $\semantics{p}$ and $\semantics{\pbar}$.
Then $\semantics{x}$ is at least as their maximum.   And this maximum is $\card{M}$.   
So  $\semantics{x}$ is at least as large as the size of any set.
\end{proof}

This concludes our discussion of the soundness of the logic.  We turn to examples.

\begin{figure}[t]
    \begin{mathframe}
        $$
        \begin{array}{c@{\qquad}c}
        \infer[{\mbox{\sc (non-empty)}}]{ \exists(p,p)}{ \existsgeq(p,\pbar)} 
        &
        \infer[{\mbox{\sc (non-empty-more)}}]{ \more(p,\pbar)}{ \forall(\pbar,p)}  
        \\   \\
        \infer[{\mbox{\sc (weak-more-anti)}}]{\existsgeq(\pbar,x)}{\more(q,p) }
        
        &       \infer[{\mbox{\sc (up)}}]{ \existsgeq(x,q)}{ \existsgeq(x,p)  & \existsgeq(x,\pbar)} 
        \\  \\
        \end{array}$$
    \end{mathframe}
    \caption{The second part of the rules for the logic in this paper.
        \label{figure-secondpart}}
\end{figure}

\begin{ex} $\forall(\xbar,x) \proves\exists(x,x)$.   Here are two different derivations:
    $$
    \begin{array}{l@{\qquad\qquad}l}
    \infer[\mbox{\sc (non-empty)}]{\exists(x,x)}
    {
        \infer[{\mbox{\sc (subset size)}}]{\existsgeq(x,\xbar)}{\forall(\xbar,x)}
    }  &
    \infer[\mbox{\sc (more-some)}]{\exists(x,x)}
    {
        \infer[{\mbox{\sc (non-empty-more)}}]{\more(x,\xbar)}{\forall(\xbar,x)}
    }
    \end{array}    
    $$
    \label{example-non-empty}
\end{ex}

\begin{ex}
    $\existsgeq(p,\pbar) \proves \existsgeq(p,q)$.   Here is a derivation:

    $$\infer[{\mbox{\sc (up)}}]{\existsgeq(p,q)}
    {\existsgeq(p,\pbar)
        &
        \infer[{\mbox{\sc (subset-size)}}]{\existsgeq(p,p)}{\infer[{\mbox{\sc (axiom)}}]{\forall(p,p)}{}}
    }
    $$
    
    \label{example-Dedekind}
\end{ex}

\begin{ex}
    $\existsgeq(p,\pbar), \existsgeq(q,p) \proves \existsgeq(q,x)$.
    Here is a derivation which quotes the result just above:
    $$ \infer[{\mbox{\sc (card trans)}}]{\existsgeq(q,x)}
    {\existsgeq(q,p) &
        \infer[{\mbox{\sc (example~\ref{example-Dedekind})}}]{ \existsgeq(p,x)}{ \existsgeq(p,\pbar)} 
    }
    $$
    \label{example-halfstar}
\end{ex}

\begin{ex}
    \label{ex-Top1}
    $\existsgeq(p,\pbar), \more(x,q) \proves \more(p,q)$:
    $$
    \infer[{\mbox{\sc (more-right)}}]
    {\more(p,q)}
    {
        \infer[{\mbox{\sc (example~\ref{example-Dedekind})}}]{ \existsgeq(p,x)}{ \existsgeq(p,\pbar)} 
        &
        \more(x,q)
    }
    $$
\end{ex}

\begin{ex}
    $\existsmore(p,\pbar), \existsgeq(\pbar,\qbar) \proves \existsmore(q, \qbar)$:

    $$
    \infer[{\mbox{\sc (weak-more-anti)}}]{\existsgeq(q,\qbar)}
    {
        \infer[{\mbox{\sc (more-left)}}]{\more(p,\qbar)}{\more(p,\pbar) & \existsgeq(\pbar,\qbar)}
    }
    $$
    \rem{
        Here is a derivation in a list format which was found by our compute implementation:
        \begin{verbatim}
        (1,there are more p than non-p,"A",[])
        (2,there are at least as many non-p as non-q,"A",[])
        (3,there are more p than non-q,"more left",[1,2])
        (4,there are at least as many q as p,"weak-more-anti",[3])
        (5,there are more p than non-p,"A",[])
        (6,there are at least as many non-p as non-q,"A",[])
        (7,there are more p than non-q,"more left",[5,6])
        (8,there are more q than non-q,"more right",[4,7])
        \end{verbatim}
        (We have improved the English.)
        This is a proof tree squashed to a linear list.
    }
\end{ex}

\begin{ex}
    $\existsgeq(p,\pbar),\more(q,\qbar) \proves\exists(p,q)$:
    
    \rem{
        $$\infer[{\mbox{\sc (more-some)}}]{\exists(p,q)}
        {
            \infer[{\mbox{\sc (more-right)}}]{\more(p,\qbar)}
            {
                \infer[{\mbox{\sc (up)}}]{\existsgeq(p,q)}
                {\existsgeq(p,\pbar)}
                &
                \more(q,\qbar)
            }
        }
        $$
    }

    $$\infer[{\mbox{\sc (more-some)}}]{\exists(p,q)}
    {
        \infer[{\mbox{\sc (more-right)}}]{\more(p,\qbar)}
        {
            \infer[{\mbox{\sc (example~\ref{example-Dedekind})}}]{\existsgeq(p,q)}
            {\existsgeq(p,\pbar)
            }
            &
            \more(q,\qbar)
        }
    }
    $$
    
    \rem{
        \begin{verbatim}
        (1,there are at least as many p as non-p,"A",[])
        (2,all p are p,"axiom",[])
        (3,there are at least as many p as p,"subset size",[2])
        (4,there are at least as many p as q,"up",[1,3])
        (5,there are more q than non-q,"A",[])
        (6,there are more p than non-q,"more right",[4,5])
        (7,some p are q,"more-some",[6])
        \end{verbatim}
    }
    \label{ex-softstricthalf}
\end{ex}

\begin{ex}
    \label{ex-int}  Here is a  rules from~\cite{moss:forDunn} called 
 ({\sc int}).
We include this rule because it is used later in this paper.
({\sc int}) says  $\exists(p,p),\more(q,\qbar) \proves\exists(q,q)$.
Here is a derivation: 
    
    $$
    \infer[{\mbox{\sc (Card-$\exists$)}}]{\exists(q,q)}
    {
        \infer[{\mbox{\sc (example~\ref{example-Dedekind})}}]{\existsgeq(q,p)}
        {\existsgeq(q,\qbar)
        }
        &
        \exists(p,p)
    }
    $$\\
    
\rem{    
    For ({\sc half}), $\existsgeq(p,\pbar),\existsgeq(q,\qbar) \proves\existsgeq(p,q)$: 
    
    $$
    \infer[{\mbox{\sc (example~\ref{example-Dedekind})}}]{\existsgeq(p,q)}
    {\existsgeq(p,\pbar)
    }
    $$
    }

\end{ex}

\rem{
\begin{ex}
    In the logic of~\cite{moss:forDunn}, {\sc (x)} is derivable from 
    {\sc (x-card)}.   The derivation uses the rule {\sc (more)} from Figure~\ref{figure-finite-setting-unsoundpart},
    and so it is not available in this paper.
    Indeed, neither   {\sc (x)}  nor {\sc (x-card)} is derivable from the other in the 
    logic of this paper.  And so we have stated both of them in Figure~\ref{figure-finite-setting-soundpart}.
\end{ex} 
}

\begin{ex} We have a derivation of $\existsmore(p,x), \existsmore(q,\xbar)\proves \phi$.
    That is, we derive a contradiction from the two assumptions 
    $\existsmore(p,x)$ and $\existsmore(q,\xbar)$.
    
    $$\infer[\mbox{\sc (x-card)}]{\phi}
    {
        \infer[\mbox{\sc (more-left)}]{\existsmore(q,q)}
        {
            \existsmore(q,\xbar) &
            \infer[\mbox{(\sc weak-more-anti)}]{\existsgeq(\xbar, q)}{\existsmore(p,x)}
        }
        & 
        \infer[\mbox{\sc (subset-size)}]{\existsgeq(q,q)}{\forall(q,q)}
    }
    $$
    \label{first-contradiction}
\end{ex}

\begin{ex}
    \label{example-first-in-logic}
    \label{second-contradiction}
    $\existsgeq(p,\pbar), \existsmore(q,p) \proves \phi$. 
    Here is a derivation:
    $$ \infer[\mbox{\sc (more-left)}]{\existsmore(q,q)}
    {\existsmore(q,p) &
        \infer[\mbox{\sc (card-trans)}]{\existsgeq(p,q)}{\existsgeq(p,\pbar)   &
            \infer[\mbox{\sc (weak-more-anti)}]{\existsgeq(\pbar,q)}{\existsmore(q,p)}
        }
    }
    $$
    And the rest at the bottom is as in Example~\ref{first-contradiction}.
\end{ex}

\begin{remark} In the remainder of this paper, $\Gamma$ denotes a 
    {\em finite\/} set of sentences.   The reason for this restriction is that the logic is not compact.
    Specifically, the set
    $$\begin{array}{lcl}
    \Gamma& = & 
    \set{\existsmore(x_0,x_1), \existsmore(x_1,x_2), \ldots, \existsmore(x_n,x_{n+1}), \ldots}
    \end{array}
    $$
    has no model, but every finite subset of it does have a model.
    So $\Gamma\models \forall(p,q)$.   But clearly we cannot have $\Gamma\proves\exists(p,q)$,
    no matter what the rules of the (sound) proof system are.
    That is, we cannot hope to define a proof system and show that
    for all (possibly infinite) $\Gamma$ and all $\phi$: 
    $\Gamma\models\phi$
    iff $\Gamma\proves\phi$.
\end{remark}

\subsection{Prior work}

\begin{figure}[t]
    \begin{mathframe}
        $$
        \begin{array}{c@{\quad}c@{\quad}c}
        \infer[{\mbox{\sc (card-mix)}}]{\forall(q,p)}{\forall(p,q) & \existsgeq(p,q)}
        &
        \infer[{\mbox{\sc (more)}}]{\more(p,q)}{\forall(q,p) & \exists(p,\qbar) }
        \\ \\
        \infer[{\mbox{\sc (more-anti)}}]{\more(\pbar,\qbar)}{\more(q,p) }
        &
        \infer[{\mbox{\sc (card-anti)}}]{\existsgeq(\qbar,\pbar)}{\existsgeq(p,q)} 
        \\   \\ 
        \infer[{\mbox{\sc (strict half)}}]{\more(p,q)}{\more(p,\pbar) & \existsgeq(\qbar,q) }
        &
        \infer[{\mbox{\sc (maj)}}]{\exists(p,q)}{\existsgeq(p,\pbar) & \existsgeq(q,\qbar) & \exists(\pbar,\qbar) }
        \\  \\
        \end{array}
        $$
    \end{mathframe}
    \caption{These rules are also part of the sound and complete system for reasoning about finite sets, 
        but they  are \emph{not} sound for the 
        wider semantics of this paper.  
        \label{figure-finite-setting-unsoundpart}}
\end{figure}

We make a short digression at this point to mention the difference between the logical system in~\cite{moss:forDunn}
and the one here.   The reader not interested in this point may safely skip ahead to Section~\ref{section-consistent-sets}.

 The syntax used in this paper is exactly the same as the one in~\cite{moss:forDunn}.
The definition of $\Model\models\phi$ is the same, too.   The difference lies in the fact that  in~\cite{moss:forDunn},
the models were taken to be finite, while in this paper they are taken to be infinite.
The rules in Figure~\ref{figure-finite-setting-unsoundpart}
are from~\cite{moss:forDunn}; they are  not   sound on infinite models.

Let us check that these are not sound when we allow the universe to be infinite.  
We give some counterexamples.  In all cases, we take $M = \omega$.

For {\sc (card-mix)}, 
let $\semantics{q} = \omega \setminus\set{0}$ and $\semantics{p} = \omega$. 

For {\sc (more)}, let $\semantics{q} = \omega \setminus\set{0}$ and $\semantics{p} = \omega$. 

For {\sc (more-anti)},  let $\semantics{q} = \set{0,1}$ and let 
$\semantics{p} = \set{0}$. 

For {\sc (card-anti)},  $\semantics{p} = \omega$ and let 
$\semantics{q} = \omega\setminus\set{0}$.

For {\sc (strict-half)}, let $ \semantics{p} = \omega$, and let $\semantics{q}$ be the set of even numbers.
Then $\card{\semantics{p}} > \card{\semantics{\pbar}}$, and 
$\card{\semantics{q}} = \card{\semantics{\qbar}}$; but 
$\card{\semantics{p}} \not> \card{\semantics{q}}$.

For {\sc (maj)},
let $$\begin{array}{lcl} \semantics{p} & = &  \set{3n: n \in N}\\
\semantics{q} & = & \set{3n + 1: n \in N}
\end{array}
$$
Then $\card{\semantics{p}} = \card{\semantics{\pbar}}  = \card{\semantics{q}} = \card{\semantics{\qbar}}$. 
Then $2\in \semantics{\pbar}\cap \semantics{\qbar}$, but $\semantics{p}\cap \semantics{q} = \emptyset$.

\begin{theorem}  \cite{moss:forDunn}  The logical system whose
    rules are shown in Figures~\ref{figure-finite-setting-soundpart} and~\ref{figure-finite-setting-unsoundpart}
    is sound and complete for the class
    of all \emph{finite models}. 
    \label{thm:finite}
\end{theorem}

\subsection{Consistent sets}
\label{section-consistent-sets}

A set of sentences in any logic  is \emph{consistent} if it is not the case that
$\Gamma\proves \phi$ for all $\phi$.   In our setting, this
is equivalent to saying that   there are no derivations from $\Gamma$ which use 
one of the two \emph{ex falso quodlibet} rules, ({\sc x}) or ({\sc x-card}).

\begin{ex}
    If $\Gamma$ is consistent and contains $\more(p,x)$, then $\Gamma\not\proves \more(q,\xbar)$ for all $q$.
    This follows from Example~\ref{first-contradiction}.
    Similarly,   Example~\ref{second-contradiction} shows that
    if $\Gamma$ is consistent and contains $\more(q,p)$, then $\Gamma\not\proves \existsgeq(p,\pbar)$.
\end{ex}

\begin{ex}  For all consistent $\Gamma$, there is some $x$ such that $\Gamma\not\proves\forall(x,\xbar)$.
    To see this, take any variable $x$.  Suppose that $\Gamma\proves\forall(x,\xbar)$.
    By $\mbox{\sc (non-empty)}$,  $\Gamma\proves\exists(\xbar,\xbar)$.
    If we also had 
    $\Gamma\proves\forall(\xbar,x)$, then $\Gamma$ would be inconsistent.
\end{ex}

\subsection{Architecture of the completeness result in this paper}
\label{section-architecture}

The main result in this paper is the Completeness Theorem~\ref{completeness}.
We shall show the following:  (1) For a given
set $\Gamma$ and sentence $\phi$,
if $\Gamma\not\proves \phi$,
then there is a model of $\Gamma$ where $\phi$ is false.

Before we prove this theorem, we show the weaker statement  (2) every consistent set $\Gamma$ has a model.
Let us explain why (2)  is weaker than (1).   If we take the special case of (1) when $\Gamma$ is $\exists(x,\xbar)$,
then we get:   if $\Gamma$ is consistent, then $\Gamma$ has a model.   So we basically get (2) as a special case of (1).
Now in logics with a rule of \emph{reductio ad absurdum} ({\sc raa}),  (2) implies (1).   Here is the argument.  
Suppose that $\Gamma\not\proves\phi$.  Then $\Gamma\cup\set{\nott\phi}$ is consistent.  (This is exactly where 
({\sc raa}) is used.  And us by (2), we have a model $\Model$ of $\Gamma$ where $\phi$ fails; 
so (1) holds.  Now our logic does not have  ({\sc raa}), since it complicates the proof search to add it.
(That is,  ({\sc raa}) is an admissible rule, as we shall prove following completeness.  But adopting  ({\sc raa}) from the start
would lead to a proof search algorithm that is more complicated that necessary.   As far as we can see, it would not be 
in polynomial time.   So the extra work that we shall do by using the ({\sc x}) rules leads to a more efficient proof search algorithm.

\section{Preliminaries}

In this section, we gather some preliminary background material which will be used later in the paper.

\subsection{Preliminary: listings of finite transitive relations}
A \emph{listing} of a set $X$  is
a sequence $x_1, \ldots, x_n$ from $X$ so that if $i\neq j$, then $x_i \neq x_j$.
Let $(T, < )$ be a finite set with a transitive, irreflexive relation.
A \emph{proper listing of $(T,<)$} is a listing of the set $T$ with the property that 
if $t_i  < t_j$, then $i < j$.
In words, the $<$-predecessors of each point are listed before it.
This is also called a \emph{topological sort}.

\begin{lemma} 
    Let $(T,<)$ be a finite set with a transitive, irreflexive relation.
    Then $(T,<)$ has a proper listing.
    \label{lemma-listing}
\end{lemma}

\begin{lemma}  Let $(T, < )$ be a finite set with a transitive, irreflexive relation.
    Let $y\in T$.   
    Then there is a proper listing of $(T,<) $ in which every  $x$  such that  $y \not\leq x$ comes before $y$ in the listing.
    \label{lemma-listing-refined}
\end{lemma}

Lemma~\ref{lemma-listing} is standard, and Lemma~\ref{lemma-listing-refined} is small refinement.
One source for the proofs is~\cite{moss:forDunn}.

\subsection{Preliminary: unions and disjoint unions}
\label{section-unions-disjoint-unions}

Beginning in Section~\ref{section-combining},
we shall need to keep track of the difference between unions and disjoint unions.
Given sets $X_1$, $\ldots$, $X_n$, we write $X_1 + \cdots + X_n$ for
$$\displaystyle{\bigcup_{i=1}^n} \  (\set{i} \times X_i).
$$
We also write $\sum_{i=1}^n\,  X_i$ for this same set.
We make use of this notation 
even when some of the sets $X_i$ are themselves unions of other sets.
For example, 
$$
\begin{array}{lcl}
(X\cup Y) + Z & = & 
\set{(1,x) : x\in X} \cup 
\set{(1,y) : y\in Y} \cup 
\set{(2,z) : z\in Z} 
\end{array}
$$
\section{The relations $\leq$, $\leq_c$, $<_c$, $\leqmore$, and $\equiv$}
\label{section-orders}

The method of proof of (1) in Section~\ref{section-architecture} is \emph{model-construction}.
We fix a consistent set
$\Gamma$  throughout this section.   It is convenient to suppress
$\Gamma$ from the notation.  We also will adopt
suggestive notation for various assertions in the logic.

\begin{definition}
    Let $\Gamma$ be a (finite) set of sentences.
    We write $x\leq y$ for $\Gamma\proves \forall(x, y)$.
    Note that $\Gamma$ is left off the notation.
    And we write $x\equiv y$ for $x\leq y \leq x$.
    
    We write $x\leqc y$  for $\Gamma \proves \existsgeq(y,x)$.
    We also write $x\equivc y$ for $x\leqc y\leqc x$, 
    and $x <_c y$ for  $x\leq_c y$ but $x\not\equivc y$.
    
    Finally, we write $x\leqmorestrict y$ if $\Gamma\proves\more(y,x)$.
\end{definition}

\begin{proposition}
    Let $\vv$ be the   set of  variables in $\Gamma$,
    and let $w$, $x$, $x'$, $y$, $y'$, $z\in \vv$.
    \begin{enumerate}
        \item If $x\leq y$, then $x\leqc y$.
        \label{tth}
        \item  $(\vv,\leqc)$ is a preorder: a reflexive and transitive relation.
        \label{tthtth}
        \item 
        $(\vv,<_c)$ is   a strict preorder.
        
        \label{prewellorder}
        
        \item If $x\leqc y$, $x\equiv x'$, and $y\equiv y'$, then $x'\leqc y'$.
        \label{tthtthtthtth}
        
        \item If $w \leq_c x  <_{more} y \leq_c z$, then $w\leq_{more} z$.
        \label{morepart}
    \end{enumerate}
    \label{prop-leqc}
\end{proposition}

\begin{proof}
    Part (\ref{tth}) uses the ({\sc subset-size}) rule.
    In part (\ref{tthtth}), the reflexivity of $\leqc$ comes from that of
    $\leq$ and part (\ref{tth}); the transitivity is by  ({\sc card-trans}).
    Part (\ref{prewellorder}) follows from the previous part.
    Part (\ref{tthtthtthtth}) uses part (\ref{tth}) and transitivity.
    Part (\ref{morepart}) uses ({\sc more-left}) and  ({\sc more-right}).
\end{proof}

\begin{remark} 
    Let us emphasize that there is a difference between $<_c$ and $<_{more}$.
    When we write $p <_c q$, we mean that 
    $$\begin{array}{l@{\quad}c@{\quad}l} \Gamma\proves\existsgeq(q,p) 
    & \mbox{and}  & \Gamma\not\proves\existsgeq(p,q).\\
    \end{array}
    $$   This is weaker than  $p <_{more} q$; recall that this last assertion means that 
    $\Gamma\proves\more(q,p)$. 
    For example, if $\Gamma$ contains just the sentence $\existsgeq(q,p)$ (and nothing else),
    then $p <_c q$ but not $p <_{more} q$.
\end{remark}

\newcommand{\pospos}{\mbox{\sf pos}}
\newcommand{\negneg}{\mbox{\sf neg}}
\newcommand{\smallsmall}{\mbox{\sf small}}
\newcommand{\largelarge}{\mbox{\sf large}}
\newcommand{\halfhalf}{\mbox{\sf half}}
\newcommand{\halforlarge}{\halfhalf\ \cup\ \largelarge}
\newcommand{\smallorhalf}{\smallsmall\,\cup\, \halfhalf}

\section{Small, large, and half}
\label{section-small-large-half}

Let $\Gamma$ be consistent in $\Sdagleq$. 
Define a partition of  variables of  $\Sdagleq$ into three classes, as follows:

\begin{enumerate}
    \item For all $p$ such that $p \leq_c \pbar\leq_c p$, put both $p$ and $\pbar$ into $\halfhalf$.
    \item For all $p$ such that $p \leqmorestrict q$ for some $q$, put $p\in\smallsmall$ and $\pbar\in \largelarge$.
    \item Put all other nouns in $\halfhalf$.
\end{enumerate}

We call this \emph{the standard partition of the nouns according to $\Gamma$}.

\begin{lemma}
    Let $\Gamma$ be consistent in $\Sdagleq$.
    The standard partition puts each $p$ and $\pbar$ into exactly one of the three class
    $\smallsmall$, $\halfhalf$, and $\largelarge$.  Moreover, this partition has the following properties:
    \begin{enumerate}
        \item[(i)]  If    ($p \leq_c \pbar$ and $\pbar\leq_c p$), then $p,\pbar\in \halfhalf$.
        \label{parti}
        \item[(ii)] If  $p \leqmorestrict q$ for some $q$, then $p\in\smallsmall$.
        \label{partii}
        \item[(iii)] $p \in \largelarge$ iff $\pbar\in\smallsmall$.
        \label{partiii}
        \item[(iv)] If $p\in \smallsmall$ and $q\leq_c p$, then $q\in \smallsmall$.
        \label{partiv}
        \item[(v)] If $p\leq_c \pbar$, then either ($p\in\smallsmall$ and $\pbar\in\largelarge$),
        or else  $p,\pbar\in\halfhalf$.  
        \label{partv}\
        \item[(vi)]   If $p\in \halfhalf$ and $q\leq_c p$, then either $q\in\smallsmall$ or $q\in\halfhalf$.
        \label{partvi}
        \item[(vii)]  If $\pbar\leq p$ (that is, $\Gamma\proves \forall(\pbar,p)$), then $\pbar\in\smallsmall$ and $p\in\largelarge$.
        
    \end{enumerate}
    \label{lemma-three-classes}
\end{lemma}

\begin{proof}
    The main thing is to check that the first two points in the definition of the standard partition cannot conflict.
    Here are the verifications.   If $p \leq_c \pbar\leq_c p$, then we cannot also have $p \leqmorestrict q$.
    (See Example~\ref{second-contradiction}.)
    We cannot have both $p \leqmorestrict q$ and $\pbar \leqmorestrict x$.
    (See Example~\ref{first-contradiction}.)
    
    Parts (i) -- (iii) are immediate from the construction.
    Part (iv) follows from Proposition~\ref{prop-leqc}.
    
    Then, we check that  (v) holds.  Suppose that $p\leq_c \pbar$ but that neither of the conditions in (v) hold. 
    By (iii), $\pbar\in\smallsmall$ and $p\in\largelarge$.  But then by (iv), 
    $p\in\smallsmall$.  And this is a contradiction.
    
    For (vi), suppose that $p\in \halfhalf$ and $q\leq_c p$, and towards a contradiction that $q\in\largelarge$.
    Then $\pbar\leq_c q$.  But also $p\leq_c \pbar$, so we have $p\equiv_c q$.   And this is a contradiction, since $q\in\largelarge$.
    (So again we are using Proposition~\ref{prop-leqc}.)
    
    Part  (vii) follows immediately from   ({\sc non-empty-more}).
\end{proof}

\subsection{Refinement}
\label{section-refine-three-classes}

We also need a refined version of the standard partition.   Let $\Gamma$ be consistent.  Fix a noun $p^*$.
Suppose that $\nott(\pbar^*\leq_c p^*)$, and also that there are no $x$ such that $\pbar^* \leqmore x$.
Then the standard partition might put $p^*\in\halfhalf$ (due to the final step in the definition of the standard partition),
and it might put $p^*\in\smallsmall$.  But it cannot put $p^*\in\largelarge$.
In case it puts $p^*\in\halfhalf$, we might want a modification which instead puts $p^*\in\smallsmall$.  

Therefore we modify the standard partition, as follows:

\begin{enumerate}
    \item For all $p$ such that $p\leq_c p^*$, put $p\in\smallsmall$ and $\pbar$ in $\largelarge$.
    \item For all $p$ such that $p \leq_c \pbar\leq_c p$, put both $p$ and $\pbar$ into $\halfhalf$.
    \item For all $p$ such that $p \leqmorestrict q$ for some $q$, put $p\in\smallsmall$ and $\pbar\in \largelarge$.
    \item Put all other nouns in $\halfhalf$.
\end{enumerate}

\begin{lemma}
    Let $\Gamma$ be consistent in $\Sdagleq$.
    Let $p^*$ be such that $\nott(\pbar^*\leq_c p^*)$, and also that there are no $x$ such that $\pbar^* \leqmore x$.
    Then the refined version of the standard partition puts
    $p^*\in\smallsmall$, and it also puts
    each of $p$ and $\pbar$ into exactly one of the three classes
    $\smallsmall$, $\halfhalf$, and $\largelarge$.  Moreover, 
    all of the points in Lemma~\ref{lemma-three-classes} hold.
    \label{lemma-three-classes-refined}
\end{lemma}

\begin{proof}
    This time, we must check that the first three points in the definition of the standard partition cannot conflict.
    Since there are no conflicts between (2) and (3), we only need to check conflicts between (1) for  a noun $p$ and its complement
    $\pbar$, conflicts between (1) and (2), and finally conflicts between (1) and (3)
    
    For all $p$, we cannot have $p\leq_c p^*$ and also $\pbar\leq_c p^*$.    This is by ({\sc up}), taking $x$ to be $p^*$.
    
    Concerning conflicts between (1) and (2), suppose that $p\leq_c\pbar \leq_c p$ and also that $p \leq_c p^*$.
    Then by Example~\ref{example-halfstar}, we have $\pbar^*\leq_c p^*$, and this contradicts our assumption on $p^*$ in this result.
    
    As for conflicts between (1) and (3), suppose that $\pbar \leqmore q$ and $p\leq_c p^*$.   
    Then by Example~\ref{ex-Top1}, $\pbar^* \leqmore q$.  So again we have a contradiction.
    
    The verification of (i) -- (vii) is nearly the same as what we saw in Lemma~\ref{lemma-three-classes}.
    That is, (i)--(iv) are immediate from the construction, and (v) and (vi) follow easily from these.
\end{proof}

The work in this section is used in Section~\ref{section52}.

\subsection{A result on infinite models and the logic of $\mbox{\em All}$, $\mbox{\em Some}$, and complemented variables}

\begin{lemma}\label{lemma-exactly-kappa}
    Let $\kappa$ be an infinite cardinal.
    Let $\Delta$ be a finite consistent set
of sentences of the form  $\exists(p,q)$ and $\forall(p,q)$.
Then $\Delta$ has a model
    with the following properties:
    \begin{enumerate}
        \item All nouns $p$  
        have the property that either $\semantics{p} = \emptyset$,  
        or $\semantics{\pbar} = \emptyset$, or 
        both $\semantics{p}$ and $\semantics{\pbar}$
        are   sets of cardinality exactly
        $\kappa$.     
        \item For all $p$ and $q$, $\semantics{p}\cap\semantics{q} \neq \emptyset$ iff $\Gamma\proves\exists(p,q)$.
        (A special case  is when $p = q$,  We see that  $\semantics{p} = \emptyset$ iff 
        $\Gamma\not\proves \exists(p,p)$.)
    \end{enumerate}
\end{lemma}

\begin{proof}
Since  our logic includes a complete syllogistic logic of $\exists$ and $\forall$, $\Delta$ is consistent in the
syllogistic logic of sentences $\forall(p,q)$, and $\exists(p,q)$.
By~\cite{logic:moss08}, $\Delta$ has a finite  model $\Model$
with the extra property that if $\Delta\not\proves \exists(p,q)$, then $\semantics{p}\cap\semantics{q}  = \emptyset$.
    Fix such a model $\Model$.
    Let $p_1$, $\pbar_1$, $\ldots$, $p_k$, $\pbar_k$ include the finite set of nouns that occur in sentences in $\Delta$.
    (Since 
    we are only dealing with finite sets of sentences, this is a finite set.)
    Let the universe $M$ of the model $\Model$ be $\set{1, \ldots, m}$.
    Let 
    $A_1$, $\ldots$, $A_k$ be a pairwise disjoint family of sets of size exactly $\kappa$.
    (For example, we might have $A_i= \kappa \times \set{i}$.)
    Then make a new model $\Nodel$ by 
    taking the universe $N$ to be $\bigcup \set{A_i  : 1 \leq i \leq k}$, and then by
    interpreting each  $p$ thus: 
    $$\begin{array}{lcl}
    \semantics{p}_{\Nodel} &  = & \bigcup \set{A_i: i\in \semantics{p}_{\Model}}.
    \end{array}
    $$
    It then follows that 
    $$\begin{array}{lcl}
    \semantics{\pbar}_{\Nodel} &  = & \bigcup \set{A_i: i\notin \semantics{p}_{\Model}}.
    \end{array}
    $$
    The assertions about  $\semantics{p}$ and $\semantics{\pbar}$ in our result are easy to check.
    
    \begin{claim}
        $\Model$ and $\Nodel$ agree on all sentences using $\forall$ and $\exists$.
    \end{claim}
    
    \begin{proof}
        Consider first a sentence $\forall(a,b)$.
        Assume that $\Model\models\forall(a,b)$.  So $\semantics{a}_\Model \subseteq \semantics{b}_\Model$.
        Let $\alpha\in \semantics{a}_\Nodel$.  So for some $i\in \semantics{a}_\Model$, 
        $\alpha\in A_i$.  But this $i$ belongs also to $\semantics{b}_\Model$.
        And therefore $\alpha\in \semantics{b}_\Nodel$.
        This for all $\alpha$ shows that $\semantics{a}_\Nodel \subseteq \semantics{b}_\Nodel$.
        In the other direction, assume that
        $\semantics{a}_\Nodel \subseteq \semantics{b}_\Nodel$.
        Let $i\in \semantics{a}_\Model$.   Since $A_i$ is infinite, it is non-empty.
        Let $\alpha\in A_i$.   Thus $\alpha\in \semantics{a}_\Nodel$.  By our assumption,
        $\alpha\in \semantics{b}_\Nodel$.    And since the sets $A_j$ are pairwise disjoint,
        we see that $i\in \semantics{b}_\Model$.   This for all $i$ shows that 
        $\semantics{a}_\Model \subseteq \semantics{b}_\Model$.
        
        We turn to the sentences $\exists(a,b)$.
        Assume that $\Model\models\exists(a,b)$, and let $i\in \semantics{a}_\Model \cap  \semantics{b}_\Model$.
        Let $\alpha\in A_i$.  Then $\alpha \in \semantics{a}_\Nodel \cap  \semantics{b}_\Nodel$.
        So $\Nodel\models \exists(a,b)$.
        Going the other way, assume that 
        $\Nodel\models \exists(a,b)$.   Let $\alpha\in  \semantics{a}_\Nodel \cap  \semantics{b}_\Nodel$.
        We have  $i\in \semantics{a}_\Model$ and $j\in \semantics{b}_\Model$
        so that $\alpha \in A_i \cap A_j$.
        But since the $A$'s are pairwise disjoint, we have $i = j$.
        Then $i \in  \semantics{a}_\Model \cap  \semantics{b}_\Model$.
        And $i$ shows us that $\Model\models \exists(a,b)$, as desired.
    \end{proof}
    
    By this claim, $\Nodel\models\Delta$.  Finally, 
    suppose that $\Gamma\not\proves\exists(p,q)$.
    Then in $\Model$, $\semantics{p}\cap\semantics{q} = \emptyset$.
    It follows that the same is true in $\Nodel$.
    Finally, suppose that $\semantics{p}_{\Nodel} \neq\emptyset$.
    As we have just seen, this means that $\Gamma\proves\exists(p,p)$.

    Finally,
    each interpretation $\semantics{p}_{\Nodel}$ is either empty or has size $\kappa$, by construction.
\end{proof}

\section{Consistent sets have models}
\label{section-model-construction}

\begin{theorem}  \label{completeness-weak}
    Every consistent (finite) set  $\Gamma$  has a model $\Model$.
    Moreover, given a partition of the nouns into three classes as in Lemma~\ref{lemma-three-classes},
    we can find a model $\Model$ with the following additional properties:
    \begin{enumerate}
        \item If $p, q\in\smallorhalf$ and $\semantics{p}\cap\semantics{q} \neq \emptyset$, then $\Gamma\proves \exists(p,q)$.
        \label{first-addendum}
        \item If $\semantics{p} \neq\emptyset$, then $\semantics{p}$ is infinite.
    \end{enumerate}
\end{theorem}

The rest of this section is devoted to the construction of a model; the verification that it works is in Section~\ref{section-verification}.
The actual completeness theorem of the logic is a stronger result,
and it appears as Theorem~\ref{completeness} below.

\paragraph{Before the construction: find the partition and the listing}
Fix $\Gamma$, and hence $\leq$, $\leq_c$, $\leqmore$, and $\equiv$.
Let 
\begin{equation}
\begin{array}{lcl}
Q & = & \set{p\in \smallsmall : \Gamma\proves \exists(p,p)} 
\end{array}
\label{eq-Q}
\end{equation}
Note that if $p\in Q$ has this property and $p\equiv_c q$, then also $q\in Q$.   This is due to the rule ({\sc int}) which is
derivable in our system.

To begin, use  Lemma~\ref{lemma-three-classes} to partition the nouns into the three sets $\smallsmall$, $\halfhalf$, and $\largelarge$.
The relation $\equiv_c$ is an equivalence relation on the nouns, and $<_c$ is an irreflexive and transitive relation on the
quotient set.  
We write $[p_i]$ for $\set{q : q \equiv_c p_i}$.   
Use Lemma~\ref{lemma-listing}, and fix a listing of 
the strict (irreflexive) order on the quotient set
$$(Q/{\!\!}\equiv_c, <_c) ,$$
say as 
\begin{equation}
\label{listing}
[p_1], \ldots, [p_n].
\end{equation}
(This sequence is finite because $\Gamma$ is a finite set of sentences.)

Our construction has several steps.  In Steps $1$, $\ldots$, $n$, we carry out a construction for
each term in (\ref{listing}).   Step $i$ builds something which is a structure (see Definition~\ref{def-models-structures} just below) but
probably not a model 
in our sense.

\begin{definition}  Let $A$ be a set of nouns.   A \emph{structure over $A$} is a 
    set $S$ together
    with interpretations $\semantics{x}$ of all $x\in A$.   (Again, the difference between a structure and a model is that a 
    structure need not interpret $\xbar$ for $x\in A$.   For that matter, $A$ could be empty; in that case, a structure over $A$ is just a set.)  We denote structures by letters $\SS$.
    \label{def-models-structures}
\end{definition}

Step $i$ will produce a structure $\Structure_i$  over $[p_i]$.
We also use the notation $\kappa_i$ for $\card{\Structure_i}$.
In Step $n+1$, 
we do the same thing for $\halfhalf$, obtaining a structure over $\halfhalf$.  
Note that $\halfhalf$ is closed under complements, unlike 
all of the classes $[p_i]$ in (\ref{listing}).   After this, we put the models $\Structure_1$, $\Structure_2$, $\ldots$, $\Structure_n$, $\Structure_{n+1}$
together, in a non-obvious way.   And at the end, we have a final step for the nouns in $\largelarge$.
Then we turn our structure into a model.

\rem{
    \subsection{Construction: when $i \leq n$ and  $\Gamma\not\proves\exists(p_i,p_i)$}
    \label{section-construction-no-provable-existence}
    
    If $\Gamma\not\proves\exists(p_i,p_i)$, then we set $\Structure_i$ to be the empty structure.
    That is, $S_i = \emptyset$, and $\semantics{x} = \emptyset$ for all $x\in [p_i]$.
    We also have $\kappa_i = 0$.
    
    We may re-order our listing so that this
    happens for some initial segment of our listing in (\ref{listing}).   The initial segment of empty structures might be 
    itself empty.  (That is, if  might the case that for all $p$, $\Gamma\proves\exists(p,p)$.)  But if 
    $\Gamma\not\proves\exists(p_i,p_i)$ and $j \leq i$, then $\Gamma\not\proves\exists(p_j,p_j)$.
    Here is the reason:  if $\Gamma\not\proves\exists(p,p)$ and $\Gamma\proves\exists(q,q)$.
    Then $\nott (q\leq_c p)$.  So we might move $p$ (or rather $[p]$) to before $q$ in out listing.
}

\subsection{Construction: 
    $i = 1, \ldots n$.}

\label{section-construction-provable-existence}

We remind the reader that the listing in (\ref{listing}) is a topological sort of the equivalence classes $[x]$
modulo $\equiv_c$ which are in $\smallsmall$ and which have the property that $\Gamma\proves\exists(x,x)$.
It is possible that $n = 0$, and in that case we skip ahead to Section~\ref{section-n-plus-one}.

Let $\kappa_1 = \aleph_0$.   Given $\kappa_{i -1}$, we first give the definition of $\kappa_i$,
and then this is used in the definition of $\Structure_i$.

\begin{definition} 
    Let $2 \leq i \leq n$.    We say that $i$ \emph{demands a larger size} if 
    for some $j < i$, $p_j <_{more} p_i$ but
    $\kappa_j  = \kappa_{i -1}$.
    (Please note that for $i = n+1$ we shall see a different definition in Section~\ref{section-n-plus-one}.)
    \label{def-demands}
\end{definition}

(Here is the point of this definition:   Suppose that $i$ demands a larger size.
We make sure that $\Structure_{i}$ is of size strictly larger than $\Structure_{i-1}$, and we
interpret all $x\in [p_i]$ by sets whose size is that of the universe $S_{i}$ of $\Structure_i$.
Thus, $\semantics{x}$ will be a set which is larger than the size of $\semantics{y}$ whenever $\existsmore(x,y)$ is in $\Gamma$.)

If $i = 1$, then as we have mentioned, $\kappa_i = \aleph_0$.
If $i >1 $ demands a larger size, set $\kappa_i = (\kappa_{i-1})^+$.  (This is the smallest infinite cardinal larger than $\kappa_i$.)
If $i>1$ does not demand a larger size, set $\kappa_i = \kappa_{i-1}$.

Continuing, let 
$$\begin{array}{lcl}
\Gamma_i & = & \set{\phi \in \Gamma : \mbox{ all variables in $\phi$ belong to $[p_i]$}}\\
\Delta_i  & = & \set{\phi\in \Gamma_i : \mbox{$\phi$ is a sentence in $\forall$ or $\exists$}} \cup \set{\exists(x,x) : x\in [p_i]}
\end{array}
$$
Let  $\Structure_i$ be a structure obtained by applying  Lemma~\ref{lemma-exactly-kappa} so that
\begin{equation}
\Structure_i \models \Delta_i
\end{equation}
\begin{equation}
\card{S_i} = \kappa_i
\end{equation}
\begin{equation}
\mbox{$\card{\semantics{x}}_{\Structure_i}  = \kappa_i$  for all $x \in [p_i]$.}
\label{cardkappai}
\end{equation}
\begin{equation}
\mbox{For all $p,q\in [p_i]$, $\semantics{p}\cap\semantics{q} \neq \emptyset$ iff $\Gamma\proves\exists(p,q)$.}
\label{nonemptiness}
\end{equation}

\rem{
    Please note that we do not regard $\Structure_i$ as a model in the exact sense of this paper.
    For us $\Structure_i$ is a set $M_i$ together with interpretations $\semantics{x}$ of all $x\in [p_i]$.
    But we do not care to interpret the complemented versions $\xbar$ of the variables $x\in [p_i]$.
    The reason is in Section~\ref{section-construction-large} below 
    we shall interpret these variables, and we cannot do this yet.
    In other words, it would be useless and therefore potentially misleading to say that $\semantics{\xbar} = M_i\setminus \semantics{x}$
    at this point.
}

\subsection{Construction:    step $n + 1$, for the nouns in $\halfhalf$}
\label{section-n-plus-one}

At this point, we take care of the nouns in $\halfhalf$.   Even if $\halfhalf = \emptyset$, our construction still might
need to increase the overall size of the model to be sure that the interpretations of variables in $\largelarge$ 
are truly larger than those in $\smallsmall$.

At step $n+1$, we have structures $\Structure_1$, $\ldots$, $\Structure_{n}$
and also $\kappa_1$, $\ldots$, $\kappa_{n}$.
Before building $\Structure_{n+1}$, we need a slightly different notion of \emph{$n+1$ demands a larger size}.
We say that 
$n+1$ \emph{demands a larger size} if either (a)
for some $j < n+1$, $p_j <_{more} p_{n+1}$, but
$\kappa_j  = \kappa_{n}$; 
or  (b) 
for some $x\in \largelarge$, there is some $j < n+1$ so that $p_j <_{more} x$, but $\kappa_j  = \kappa_{n}$;
or (c) $\kappa_n = 0$.
As before, if $n+1$ demands a larger size, set $\kappa_{n+1} = (\kappa_n)^+$; otherwise, set $\kappa_{n+1} = \kappa_n$.
(In case (c) just above, we take $\kappa_{n+1} = \aleph_0$.)

Then build $\Structure_{n+1}$ exactly in Section~\ref{section-construction-provable-existence},
using $\Delta_{n+1}$ and $\kappa_{n+1}$.
Note that for $h\in \halfhalf$, both $h$ and $\hbar$ are interpreted in $\Structure_{n+1}$.

\subsection{Construction: combining $\Structure_1$, $\ldots$, $\Structure_{n+1}$ into a structure $\Structure$}
\label{section-combining}

Let 
\begin{equation}
\begin{array}{lcl}
S & = &  S_1 + \cdots + S_{n+1} + \Gamma_{\exists}\\
\end{array}
\label{eq-M}
\end{equation}
where $\Gamma_{\exists}$ is the set of sentences $\exists(a,b)$ which belong to $\Gamma$.
(The $+$ notation is explained in Section~\ref{section-unions-disjoint-unions}.)
$S$ is 
made into a structure called $\Structure$, as follows.  For $x\in\smallorhalf$,
\begin{equation}
\begin{array}{lcll}
& \semantics{x} & = &  \sum_{i} \bigcup_{y\in [p_i]} \set{
    \semantics{y}_{\Structure_i}  : y \leq x}
+   \existsset_x\\ 
\mbox{where} 
& \existsset_x & = &  \set{\exists(y,z)\in \Gamma: y\leq x \mbox{ or } z \leq x} \\
\end{array}
\label{eq-x}
\end{equation}

Once again, note that $\Structure$ is a structure, not necessarily a model.
An overall requirement of our models is that they interpret $x$ and $\xbar$ as complements for all $x$.
This is probably not the case for $x\in \halfhalf$, and so this is why we need Section~\ref{section-construction-half}
just below.

\subsection{Construction: taking care of the variables in $\halfhalf$}
\label{section-construction-half}

For $x,\xbar\in\halfhalf$, there is more to do.   The sets $\semantics{x}$ and $\semantics{\xbar}$ are
disjoint, and their cardinalities are the same.   
But it need not be the case that $\semantics{x}\cup\semantics{\xbar}$ is all of $S$.
(See~\cite{moss:forDunn} for the easy details.)
And now we use Lemma~\ref{lemma-use-zorn} just below to define the semantics in this case.

\begin{lemma}
    Let $\kappa$ be an infinite cardinal, 
    let $W$ be a set of size $\kappa$.
    Let $N$ be a set of nouns which is closed under complements.
    For each $x\in N$, let $\pair{x}$ be a subset of $W$.
    We assume two properties:
    \begin{enumerate}
        \item $\pair{x}$ and $\pair{\xbar}$ are disjoint for all $x$.
        \item If  $x\leq y$, then $\pair{x}\subseteq \pair{y}$.
    \end{enumerate}
    Then, there are sets $\pairalt{x}$ for all $x\in N$ with the 
    following properties:
    \begin{enumerate}
        \item $\pairalt{x}$ and $\pairalt{\xbar}$ are disjoint for all $x$.
        \item If  $x\leq y$, then $\pairalt{x}\subseteq \pairalt{y}$.
        \item $\pair{x}\subseteq \pairalt{x}$.
        \item $\pairalt{x}\cup\pairalt{\xbar} = W$.
    \end{enumerate}
    Moreover, if $\pair{x}$ has size $\kappa$ for all $x$, then so does $\pairalt{x}$.
    \label{lemma-use-zorn}
\end{lemma}

\begin{proof}
    We consider the family of all functions $\pairalt{\ }: N \to \mathcal{P}(W)$
    where $\mathcal{P}(W)$ is the set of subsets of $W$, 
    and where we satisfy (1)--(3) in our result.   This family ordered by pointwise inclusion.   The resulting poset is closed under unions of chains, and so we may use Zorn's Lemma.   Let $\pairalt{\ }$ be a maximal element.
    We claim that for all $x$,  $\pairalt{x}\cup\pairalt{\xbar} = W$.
    
    Suppose towards a contradiction that this is false.
    Fix some $x$ so that  $\pairalt{x}\cup\pairalt{\xbar} \neq W$, and also fix $w\notin \pairalt{x}\cup\pairalt{\xbar}$.
    If for some $y \geq x$, $w\in \pairalt{y}\cup\pairalt{\ybar} $, then $w$ can only belong to one of those sets 
    (since they are disjoint).   Let us assume that $w\in\pairalt{y}$ and that $y\geq x$.
    (Other options include:   $w\in\pairalt{\ybar}$ and that $y\geq x$; or 
    $w\in\pairalt{y}$ and that $y\geq \xbar$.   All of the details are similar in these cases, mutatis mutandis.)
    In this case, put $w\in \pairalt{z}$ for $z \geq y$.
    This gives a larger $\pairalt{\ }$ function.    And if there is no $y$ like this, we can simply put 
    $w\in \pairalt{z}$ for $z \geq y$.  Either way, we have a larger $\pairalt{\ }$ function.
    So we have our contradiction.
    
    This shows the claim, and hence condition (4) holds.  The last assertion (on cardinalities) is easy.
\end{proof}

We apply Lemma~\ref{lemma-use-zorn}
to interpret all of the nouns in $\halfhalf$.  We  take 
$\kappa = \kappa_{n+1}$,
$W$ to be $M$ from (\ref{eq-M}),  
$N$ to be $\halfhalf$, 
and $\pair{x} = \semantics{x}$ from (\ref{eq-x}).
The properties which we have seen to insure that all of the assumptions in Lemma~\ref{lemma-use-zorn} hold.
Also, if $x\in\halfhalf$, then $\Gamma\proves \exists(x,x)$ by ({\sc non-empty}).
Thus, $\semantics{x}_{\SS_{n+1}} \neq \emptyset$ for all such $x$.  The same holds when we move from $\SS_{n+1}$ to $\SS$.
Thus the size of $\semantics{x}$ in $\Nodel$ is exactly $\kappa_{n+1}$.

\subsection{Construction: the remaining nouns, and the overall model $\Model$}
\label{section-construction-large}

Up until this point, we have not interpreted the nouns which are in $\smallsmall$ but not in 
the set $Q$ of (\ref{eq-Q}).   For each $x$ in that set, we take $\semantics{x} = \emptyset$.
So we now have interpreted all nouns in $\smallorhalf$.

For $x\in\largelarge$, we would like to set
\begin{equation}
\begin{array}{lcl}
\semantics{x} & = & S\setminus \semantics{\xbar}.
\end{array}
\label{eq-semantics-large-x}
\end{equation}
However, we also would like the size of $\semantics{x}$ to be at least that of $S$.
If any $y\in\smallorhalf$ has $\cardparentheses{S\setminus\semantics{y}_{\SS}} < \card{S}$, then add a copy of $S$
to the universe without expanding the size of any variable in $\smallorhalf$.
That is, replace $S$ by $S+S$ with all $\semantics{y}$ taken from the copy on the left.
Then interpret $\semantics{x}$ by (\ref{eq-semantics-large-x}) above, 
For the nouns $h\in\halfhalf$, we need to chose one of $h$ or $\hbar$ to include the new points.
And we need to be sure that any inequalities $h_1 \leq h_2$ are respected.
So on top of all this, we need to appeal to Lemma~\ref{lemma-use-zorn} again.

At long last, we have interpreted all nouns: 
we have arranged that $\semantics{\xbar} = S\setminus\semantics{x}$ for $x\in \smallsmall$. 
It follows that the same fact is true for $x\in\largelarge$.  And for $x\in\halfhalf$, we have arranged for this in Section~\ref{section-construction-half}.
So we have a bona fide model.    We write this model as $\Model$,
we change the name of the universe from $S$ to $M$,
and we let $\kappa = \card{\Structure}$.
The proof that $\Structure\models\Gamma$
comes next.

\section{Verifying the properties of the model}
\label{section-verification}

We must verify that $\Model\models\Gamma$.  For this, we split the argument into a number of claims.

\begin{claim}
    Let $x\notin Q$, where $Q$ is as in (\ref{eq-Q}).  Then $\semantics{x} = \emptyset$.
    \label{claim-not-in-Q}
\end{claim}

\begin{proof} This was directly established in Section~\ref{section-construction-large}. 
\end{proof}

\begin{claim} If $i \leq j$, then $\kappa_i \leq \kappa_j$.
    $\kappa_{n+1}$ is an infinite cardinal.  
\end{claim}

\begin{proof}
    An easy induction shows the monotonicity assertion in our claim.
    Suppose that some $j \leq n$,
    $j$ demands a larger size.  Let $j$ be least with this property.
    Then $\kappa_j$ is infinite, and so $\kappa_{n+1}$ also will be.
    If no $j \leq n$ demands a larger size, then $\kappa_n = 0$.
    So by condition (c) in Section~\ref{section-n-plus-one},
    we see that $\kappa_{n+1} = \aleph_0$.
\end{proof}

\begin{claim}  \label{cl1133}
    If $z\leq w$ are in $\smallorhalf$, then $\semantics{z} \subseteq \semantics{w}$.
\end{claim}

\begin{proof}
    First, if $z\notin Q$, then $\semantics{z} = \emptyset$, and we are done.
    So we shall assume that $z\in Q$.   If $w\notin Q$, then $z\notin Q$.
    For if $z\in Q$ and $z \leq w$, then it is easy to use the logic to see that $\Gamma\proves\exists(w,w)$;
    this is a contradiction to $w\notin Q$.   Thus, we only need to verify this claim when $z$ and $w$ are not in $Q$.

    In $\Structure$, we have $\semantics{z} \subseteq \semantics{w}$
    by the definition of the semantics in (\ref{eq-x}) and the (easy to check) fact that $\existsset_z \subseteq \existsset_w$.
    So even when $x$ and $y$ are in $\halfhalf$ and we 
    use Lemma~\ref{lemma-use-zorn} to expand their interpretations in $\Model$, we still see that  $\semantics{z} \subseteq \semantics{w}$.
\end{proof}

\begin{claim}
    \label{cardclaim}
    Let $1 \leq i \leq n+1$.  If $x\in [p_i]$,  then $\card{\semantics{x}} =  \kappa_i $.
\end{claim}

\begin{proof} 
    If $\Gamma\not\proves\exists(x,x)$, then also $\Gamma\not\proves\exists(p_i,p_i)$
    by ({\sc int}); see Example~\ref{ex-int}.
    In this case $\kappa_i = 0$, $\semantics{x} = \emptyset$, and our result follows.
    The more interesting case is when $\Gamma\proves\exists(p_i,p_i)$.
    According to (\ref{eq-x}),  $\semantics{x}$ is the disjoint union of a family of sets $\bigcup \set{
        \semantics{y}_{\Structure_i}  : y \leq x}$ and a finite set $\existsset_x$.
    Since all the interpretations are infinite, we ignore the finite set $\existsset_x$.

    If $y \leq x$, then the unique $j$ such that $y\in [p_j]$ is at most $i$.
    This is by ({\sc subset-size}).
    Thus $\card{\semantics{y}_{\SS_j}} = \kappa_j \leq \kappa_i$.
    Hence $\semantics{x}$ in our model $\Model$ is a finite union of sets of size $\leq \kappa_i$.
    So 
    $\cardparentheses{\semantics{x}} \leq \kappa_i$.
    But in $\Structure_i$, $\cardparentheses{\semantics{x}} = \kappa_i$.
    (For $i \leq n$, this is by (\ref{cardkappai}).   For $i = n+1$, this was noted at the end of Section~\ref{section-construction-half}.)
    So the same is true in $\Model$.
\end{proof}

\begin{claim}
    Let $x, y \in \smallorhalf$.   If $\semantics{x} \cap \semantics{y} \neq \emptyset$, then
    $\Gamma\proves \exists(x,y)$.
    \label{nonemptiness-claim}
\end{claim}

\begin{proof}
    Let $i$ and $j$ be such that $x\in [p_i]$ and $y\in [p_j]$.
    The interpretations $\semantics{x}$ and $\semantics{y}$ in $\Model$ 
    are disjoint unions, and so there is some $k$ such that 
    $$\begin{array}{lcl}
    (\semantics{x} \cap S_k)
    \cap
    (\semantics{y} \cap S_k)
    & \neq & \emptyset
    \end{array}
    $$
    Let $\alpha$ belong to the set on the left.
    There are $p \leq x$ and $q\leq y$ in $[p_k]$ such that $\alpha \in \semantics{p}_{\Model_k} \cap  \semantics{q}_{\Model_k}$.
    But then $ \semantics{p}_{\Model_k} \cap  \semantics{q}_{\Model_k} \neq \emptyset$.
    And so by (\ref{nonemptiness}), $\Gamma\proves \exists(p,q)$.  Together with  $p \leq x$ and $q\leq y$,
    we have our result.
\end{proof}

\begin{claim}
    \label{largeclaim}
    Let $x\in\largelarge$.  Then $\card{\semantics{x}} = \kappa_{n+1}$.
    If in addition, $y \leqmore x$, then 
    $\card{\semantics{y}} < \kappa_{n+1}$.
\end{claim}

\begin{proof}
    The fact that  $\card{\semantics{x}} = \kappa_{n+1}$
    follows from what we did in Section~\ref{section-construction-large}.
    
    Note that $y$ belongs to $\smallsmall$
    by Lemma~\ref{lemma-three-classes}.
    Let $ j \leq n$ be such that   $y\in [p_j]$.
    
    Suppose that $\kappa_j = \kappa_{n}$.
    Then $n+1$ demands a larger size (due to $x$ and $y$), and so $\kappa_j = \kappa_n < \kappa_{n+1}$.
    And if $\kappa_j <  \kappa_{n}$, then of course 
    $\kappa_j < \kappa_n \leq \kappa_{n+1}$.
    Either way,  $\kappa_j < \kappa_{n+1}$.
    Now our result follows from Claim~\ref{cardclaim}:
    $\card{\semantics{y}} =  \kappa_j < \kappa_{n+1}$.
\end{proof}

\begin{claim}
    $\Model$ satisfies all $\exists$ sentences in $\Gamma$.
    \label{claim-all-exists}
\end{claim}

\begin{proof}
    Consider first a sentence $\exists(x,y)\in \Gamma$, where $x$ and $y$ both belong to $\smallorhalf$.
    The sentence $\exists(x,y)$ itself belongs to
    $\existsset_x \cap \existsset_y$, hence to 
    $\semantics{x}\cap\semantics{y}$ in $\Model$.
    Therefore, $\Model\models \exists(x,y)$.
    
    Now let us consider  sentences  $\exists(x,y)\in \Gamma$, where $x\in (\smallorhalf)$ and $y\in\largelarge$.
    Then $\exists(x,y)\in \semantics{x}$.  
    Suppose towards a contradiction that $\exists(x,y)\in \semantics{\ybar}$.
    Note that $\ybar\in\smallsmall$.  In view of our definition of $\existsset_{\ybar}$,
    either $x\leq \ybar$ or $y\leq\ybar$.
    Either of these alternatives, together with  $\exists(x,y)\in \Gamma$ shows that $\Gamma$ is inconsistent.

    Finally, we consider  $\exists(x,y)$, where $x,y\in\largelarge$.
    In this case, we argue that $\Model\models \exists(x,y)$, even when this sentence does not belong to $\Gamma$.
    
    As shown in Claim~\ref{largeclaim},
    the nouns in $\smallsmall$ are interpreted by sets whose size is strictly smaller than
    the size of the universe, $\kappa_{n+1}$.    
    Thus $\semantics{\xbar} \cup\semantics{\ybar}$ has size $<\kappa$.
    So the size of $M \setminus (\semantics{\xbar} \cup\semantics{\ybar})$ is $\kappa_{n+1}$.    In particular, it is non-empty.    
    This is to say that $\semantics{x} \cap \semantics{y}$  is non-empty.    
\end{proof}

\begin{claim} 
    \label{115}
    $\Model$ satisfies every $\forall(x,y)$ sentence  in $\Gamma$.
\end{claim}

\begin{proof} 
    We consider a number of cases.
    
    \begin{enumerate}
        \item $x,y\in \smallorhalf$.   This case follows by Claim~\ref{cl1133}.
        
        \item \label{needtojustify}
        $x\in\smallsmall$ and $y\in\largelarge$.
        Let us check that $\semantics{x} \cap\semantics{\ybar} = \emptyset$, since this implies the result.
        Note that $\ybar\in\smallsmall$.  Suppose towards a contradiction that 
        $\semantics{x} \cap\semantics{\ybar} \neq \emptyset$.
        By 
        Claim~\ref{nonemptiness-claim}, $\Gamma\proves \exists(x,\ybar)$.  But then $\Gamma$ is inconsistent.
        This contradiction shows that $\semantics{x} \cap\semantics{\ybar} = \emptyset$.
        
        \item $x\in\halfhalf$ and $y\in\largelarge$.  Consider $\ybar$ and $\xbar$, and then appeal to Case 2 above.
        \item  $x\in\largelarge$.  Then we also have $y\in\largelarge$.   Consider $\ybar$ and $\xbar$, and then appeal to Case 1.
    \end{enumerate}
    
    This completes the proof.
\end{proof}

\begin{claim} 
    \label{115b}
    $\Model$ satisfies every $\more(x,y)$ sentence  in $\Gamma$.
\end{claim}

\begin{proof}
    Consider a sentence $\more(x,y)\in \Gamma$.  So $y <_{more} x$.
    Thus $y <_c x$ as well.    
    We first consider the case that $x\in \smallsmall$.  
    In our listing  (\ref{listing}), we must list $[y]$ before $[x]$.
    We must check that our construction insures that $\card{\semantics{y}} < \card{\semantics{x}}$.
    That is, let $j < i$ be such that $x \in [p_i]$ and $y\in [p_j]$. 
    Let $k$ be least such that $j < k \leq i$ and $k$ demands a larger size.  There must exist such a $k$; it is the least 
    number such that there is some $w$ in $[p_w]$ with $\Gamma\proves\more(w,y)$.   
    Our construction arranges that
    $\kappa_i \geq \kappa_{k} = (\kappa_{k-1})^+  > \kappa_k \geq \kappa_j$.
    And by Claim~\ref{cardclaim}, 
    we see that $\card{\semantics{x}} > \card{\semantics{y}}$.

    We next turn to the case $x\in\halfhalf$.
    Here we must have $y\in \smallsmall$, by Lemma~\ref{lemma-three-classes}.
    Let $j \leq n$ be such that $y\in [p_j]$.
    If $\kappa_n  > \kappa_j$, then 
    $$\begin{array}{lclclclclcl}
    \card{\semantics{x}} & =  & \kappa_{n+1} & \geq & \kappa_n & > & \kappa_j & = & 
    \card{\semantics{y}}.
    \end{array}
    $$
    If $\kappa_n  = \kappa_j$, then $n+1$ demands a larger size.  
    So $\kappa_{n+1} > \kappa_n$, and we have the same fact: $\existsmore(x,y)$ is true in $\Model$.
    
    Finally,  the case of $x\in\largelarge$ follows immediately from Claim~\ref{largeclaim}.
    
    This completes the proof.
\end{proof}

\begin{claim} 
    \label{115c}
    $\Model$ satisfies every $\existsgeq(x,y)$ sentence  in $\Gamma$.
\end{claim}

\begin{proof} 
    First, when $x\in\halforlarge$, $\card{\semantics{x}} = \kappa_{n+1}$.
    (See Claims~\ref{cardclaim} and~\ref{largeclaim}.) 
    In this case our result follows.
    
    We are left with $x\in\smallsmall$.  In this case, $y\in\smallsmall$, too, by Lemma~\ref{lemma-three-classes}, part~(iv). 
    Let $i$ and $j$ be such that $x\in [p_i]$ and $q\in [p_j]$.  
    Since $y \leq_c x$, have $x\equiv_c y$ or $x < _c y$.   
    In the first case, $i = j$ and so 
    $\card{\semantics{y}}= \card{\semantics{x}}$.
    In the second case, $[p_i]$ comes before $[p_j]$ in the listing (\ref{listing}); this is by our definition of a listing.
    So $i\leq j$, and thus $\kappa_i \leq \kappa_j$.
    Now our result follows by Claim~\ref{cardclaim}.
\end{proof}

\subsection{Completing the proof of Theorem~\ref{completeness-weak}}
The work in the previous part of this section constitutes a proof of Theorem~\ref{completeness-weak}.  
Although we did not discuss the additional properties of the model which were stated in that theorem,
the construction has arranged them.

\begin{remark} 
    At this point, we need an important remark on the entire construction of this section.
    At various places, we needed to know that various numbers demanded a larger size.
    We never used the assumption that any numbers did not demand a larger size.
    (We shall need this in Section~\ref{section52} below.)   
    And examining all of our claims, we see that the 
    verifications all go through when \emph{more} numbers demand a larger size.
    
    This remark gives a certain flexibility to our construction which we shall exploit in
    Section~\ref{Sdagleq-first-case}.
    \label{remark-change-construction}
\end{remark}


\begin{remark}
	Let $ \Gamma $ be a consistent set of sentences and suppose that  $ (x\leq_c y)$. Claim~\ref{115c} showed that $ \Model $ satisfies $\existsgeq(y,x)  $. Since $ \Model $ satisfies $\existsgeq(y,x)  $, we also have $ \Model \not\models  \more(x,y)  $. In this case, we have a model whenever any numbers do not demand a larger size (see Claims \ref{cardclaim} and \ref{115c}). We shall use this remark in Section \ref{section52}.
\end{remark}

\section{The full result}

This section refines Theorem~\ref{completeness-weak} and also proves the completeness of the logic.

\begin{theorem} \label{completeness}
    If $\Gamma\not\proves\phi$, then there is a model of $\Gamma$ in which
    $\phi$ is false.
\end{theorem}

\paragraph{Caution}  
We review the point made in Section~\ref{section-architecture}.
It is very tempting at this point to argue as follows for the completeness of the logic. 
``If $\Gamma\not\proves\phi$, then $\Gamma\cup\set{\phibar}$ is consistent, where $\phibar$ is the
negation of $\phi$ which we saw in Definition~\ref{semanticnegation}.  Thus by 
Theorem~\ref{completeness-weak}, this set has a model.   Such a model will satisfy $\Gamma$ and falsify $\phi$.''
The problem with this is that we do not have \emph{reductio ad absurdum} in the logic.  And so we are not entitled
to say that  $\Gamma\cup\set{\phibar}$ is consistent.   In fact, this set \emph{is} consistent, and this follows from the work in this section.   We \emph{could} try to prove the consistency of $\Gamma\cup\set{\phibar}$ proof-theoretically
rather than semantically, but this seems much harder.

\paragraph{Assumptions at this point}
Here is how we show the completeness of the logic.
We split into cases according to $\phi$.    We show that either $\Gamma\proves \phi$, or that $\Gamma$ is inconsistent,
or else there is a model of $\Gamma$ where $\phi$ is false.

\paragraph{Reminder}  The model construction involves many choices that can be manipulated.
First, the order of the atomic sentences makes a difference in the three-fold partition in Lemma~\ref{lemma-three-classes}.
And even after that partition is determined, the particular listing which we use in (\ref{listing}) also affects the model in a big way.

\subsection{The first case: $\phi$ is of the form $\existsgeq(x,y)$}  \label{Sdagleq-first-case}

There are a number of cases.
Our first is when $\xbar \leq_c x$.   In this case,  $\Gamma\proves \phi$.   And so we contradict the assumption in this case.
Thus, we assume that  $\nott (\xbar\leq_c x)$ in what follows.   

The second case is when $\xbar \leqmore z$ for some $z$.    In this case, we again have 
$\Gamma\proves \phi$.    So we assume that for all $z$,  $\nott (\xbar\leqmore z)$.

Thus, we may refine the standard partition
as in Lemma~\ref{lemma-three-classes-refined}: we get a partition with $x\in\smallsmall$.

We further break into cases as to whether $y\in\smallsmall$ or $y\in\halforlarge$.

If both $x$ and $y$ are in $\smallsmall$, 
let $i,j\leq n$ be such that $x\in [p_i]$ and $y\in [p_j]$.
We take a listing whether the class of $[p_i]$ comes before $[p_j]$.   
We can do this by Lemma~\ref{lemma-listing-refined}.   And then in our construction, when we come to Step $j$,
we decide that $j$ demands a bigger size.   Thus we arrange that $\kappa_j  > \kappa_{j-1} \geq \kappa_i$.
This builds a model where $\semantics{x}$ is a set of smaller size than $\semantics{y}$.

If $x\in \smallsmall$ and $y\in\halforlarge$, then change the construction so that $n+1$ demands a larger size.
(This might be true according to our work before, but even if $n+1$ did not demand a larger size,
we can insist on it:  see Remark~\ref{remark-change-construction}.)
We  get a model with 
$\semantics{y}$ a set of strictly larger size than $\semantics{x}$.
So $\existsgeq(x,y)$ is false in that model.

\subsection{The next case: $\phi$ is of the form $\more(x,y)$}
\label{section52}

Consider the standard partition according to $\Gamma$.  If $x$ and $y$ are both in $\halforlarge$, we 
have a model where 
$\card{\semantics{x}} = 
\card{\semantics{y}}$, as desired.
If $x\in\smallsmall$ and $y\in\halforlarge$, 
we get a model where $\card{\semantics{x}} 
\leq
\card{\semantics{y}}$.

If $x\in\halfhalf$ and $y\in \smallsmall$, then we have two further cases.
In case $x\in\halfhalf$ because $x \leq_c \xbar \leq_c x$,
of even if just $\xbar \leq_c x$, we argue as follows.
Since  $y\in \smallsmall$ in the standard partition,
there is some $z$ such that $y <_c z$.  So by
Example~\ref{ex-Top1}, $y <_c x$.
And if $\nott(\xbar\leq_c x)$, then we may refine the standard partition
as in Lemma~\ref{lemma-three-classes-refined}.
At this point, we have a partition where $x$ and $y$ are both in $\smallsmall$, and we treat this below.

If $x\in\largelarge$ and $y\in \smallsmall$, then we have $p$ and $q$ such that 
$\xbar <_c p$ and $y <_c q$.  
By ({\sc weak-more-anti}), $q\leq_c x$.  And so $y <_c x$ as well.   This contradicts
the assumption in this section that $\nott(y <_c x)$.

We turn to the case when
$x$ and $y$ are both in $\smallsmall$,  either in the standard partition or in a refinement thereof.
(It makes no difference.)
By the hypothesis in this section, we do not have $y <_{more} x$.
If $x\equiv_c y$, then $[ x] = [y]$.   We thus get a model where 
$\card{\semantics{x}} = 
\card{\semantics{y}}$.
We have two more cases.  First, consider what happens when
$[y] \not{\!\!\!}{\leq}_c [x]$.   We chose a listing of the $\equiv_c$-classes that puts $x$ first.
And then we get a model where  $\card{\semantics{x}} <
\card{\semantics{y}}$.
Finally, we have the case when 
$y \leq_c x$.   So $[y]$ precedes $x$ in the listing.
In this case, we must alter the listing even further.   Some of the classes between $[y]$ and $[x]$ might
demand a larger size.   We move those (in order) to after $[x]$ in the listing.   We must be sure
that all of the classes which are $<_c [x]$ still come before it in the listing.   
In other words, the classes $[z]$ between $[y]$ and $[x]$ which demand a larger size may be moved to after $[x]$
without falsifying the key property of the listing.   The reason is that if such a class $[z]$ were problematic,
then we would contradict $\nott (y \leqmore x)$.
In this case, we get a model where $\semantics{x}$ and $\semantics{y}$ have the same size. 

\rem{
\begin{remark}  One should read the proof above carefully, with attention to the definition of
    \emph{demands a larger size},
    Definition~\ref{def-demands}.   So since that definition is subtle, we have to check that all of this is correct.
\end{remark} 
}

\subsection{The next case: $\phi$ is of the form $\forall(x,y)$}

We break into a number of subcases.   In each case, we take the model $\Model$ of $\Gamma$ from
Section~\ref{section-model-construction} and modify it by adding a point.
$\Model$ has the property that the interpretation of everything
is either empty or infinite.
We shall be adding one point, call it $*$, to this model.   The point is added in such a way that
it makes $\forall(x,y)$ false.  
The addition of one point to any set doesn't change the truth of any sentences involving cardinality.
Also, none of the $\exists(p,q)$ sentences changes truth value when a point is added to a model.
But we shall be interested to check that the $\forall(p,q)$ sentences from $\Gamma$ are true even after the point is added.

\begin{enumerate}
    \item $x\in\smallorhalf$, $y\in\smallorhalf$.  Add one fresh point $*$ to  $\semantics{z}$
    for all $z$ such that $x\leq z$.   
    
    Here are the details on the $\forall(p,q)$ sentences in $\Gamma$.   Review the argument in Claim~\ref{115}.
    The only case which we must consider is when $p\in\smallsmall$ and $q\in\largelarge$.
    If the new point $*$ belongs to the interpretation of $p$ and $\qbar$, then
    $x\leq p$ and $x\leq \qbar$. 
    But $\qbar \leq \pbar$.   And so $x\leq \pbar$.  Together with $x\leq p$, we see that $x\leq \xbar$.
    And so $\Gamma\proves\forall(x,y)$.   
    This contradiction shows that  indeed $\forall(p,q)$ is true after $*$ is added.  Thus we obtain a model
    of $\Gamma$ which falsifies $\phi$, as desired.

    \item $x\in\halforlarge$ and $y\in\halforlarge$:  replace $x$ and $y$ by $\ybar$ and $\xbar$, and apply
    the last case.   We get a model falsifying $\forall(\ybar,\xbar)$, and thus a model falsifying $\forall(x,y)$.
    
    \item $x\in\smallsmall$, $y\in\largelarge$. 
    This time, we add a single fresh point $*$ to $\semantics{z}$ whenever $x\leq z$ or $\ybar \leq z$.
    We assume that $*$ gets added to $\semantics{a}$, since otherwise $\semantics{a} \subseteq\semantics{b}$ after the addition.
    If $*$ also gets added to  $\semantics{b}$, then again $\semantics{a} \subseteq\semantics{b}$.
    
    We have two cases as to why  $*$ gets added to $\semantics{a}$:  $x\leq a$, and $\ybar\leq a$.
    Since $*$ was not added to $\semantics{b}$, it was added to $\semantics{\bbar}$.  (Notice that $\bbar\in\smallsmall$.)
    So again, we have two cases:   $x\leq \bbar$, and $\ybar \leq \bbar$.  The second is equivalent to $b\leq y$.
    
    We shall examine all four cases and show that in each of them we have $x\leq y$.  This contradicts the overall assumptions in this section.
    
    (a)  $x\leq a$ and $x\leq \bbar$.   So $x\leq a \leq b$, and also $x\leq \bbar$.  Thus $x\leq \xbar$.  As a result, $x\leq y$.

    (b) $x\leq a$ and $b\leq y$.   So $x\leq a \leq b \leq y$.   Thus $x\leq y$.   
    
    (c)  $\ybar \leq a$ and $x\leq \bbar$.   The $\ybar \leq a \leq b\leq \xbar$.  So we have $\ybar\leq \xbar$, and thus $x\leq y$.
    
    (d) $\ybar\leq a$ and $b\leq y$.  This time $\ybar \leq y$.  This implies $x\leq y$.
    
\end{enumerate}

\subsection{The last case: $\phi$ is of the form $\exists(x,y)$}
\label{section-exists-full-result}

Consider the standard partition of the variables according to $\Gamma$.
When $x$ and $y$ are in $\smallorhalf$, the model  
$\Model$ constructed in
the previous section works: see part~\ref{first-addendum} of Theorem~\ref{completeness-weak}.
When $x$ and $y$ are both in $\largelarge$, or when one is in $\halfhalf$ and the other in $\largelarge$
are easy: in these cases $\Gamma\proves \exists(x,y)$.
(See Example~\ref{ex-softstricthalf}.)

We are left with the case that one of the variables, say $x$, is in $\smallsmall$, and the other one,  $y$, is in $\largelarge$.
Both $x$ and $\ybar$ belong to $\smallsmall$.
Further, we cannot have $\Gamma\proves\existsmore(x,\ybar)$, since this implies $\Gamma\proves \exists(x,y)$,
using $\mbox{({\sc more-some})}$.  
Since  $\Gamma\not\proves\existsmore(x,\ybar)$, we know from our work
in 
Section~\ref{section52} that there is a model of $\Gamma$ where 
$\card{\semantics{x}} \leq
\card{\semantics{\ybar}}$. 

If $\Model\models \forall(y,\ybar)$, then again we are done.
So we shall assume that $\Model\models\exists(\ybar,\ybar)$.   And since $\ybar\in\smallsmall$,
Theorem~\ref{completeness-weak} tells us that $\Gamma\proves\exists(\ybar,\ybar)$.

We modify $\Model$ to obtain a different model to be called $\Nodel$.
The two models have the same set of points:  $N = M$.
The interpretation function of $\Nodel$ will be written $\semantics{ \ }^*$.

For $p\in \smallsmall$, 
\begin{enumerate}
    \item If $\ybar \leq p$, then $\semantics{p}^* = \semantics{p} \cup \semantics{x}$
    and  $\semantics{\pbar}^* = \semantics{\pbar} \cap \semantics{\xbar}$.
    \item If $\nott (\ybar\leq p)$, then $\semantics{p}^* = \semantics{p}$ and
    $\semantics{\pbar}^* = \semantics{\pbar}$.
\end{enumerate}
These clauses also define $\semantics{p}^*$ for $p\in\largelarge$.

For $p\in \smallorhalf$,  note that we cannot have $\ybar \leq p, \pbar$.
(For if we did, then $\ybar\leq y$,
and we contradict our assumption that $\Gamma\proves\exists(\ybar,\ybar)$.)
\begin{enumerate}
    \item If $p\in\halfhalf$ and $\ybar \leq p$, then 
    $\semantics{p}^* = \semantics{p} \cup \semantics{x}$
    and  $\semantics{\pbar}^* = \semantics{\pbar} \cap \semantics{\xbar}$.
    \item
    If $p\in\halfhalf$ and $\ybar \leq \pbar$, then 
    $\semantics{\pbar}^* = \semantics{\pbar} \cup \semantics{x}$
    and  $\semantics{p}^* = \semantics{p} \cap \semantics{\xbar}$.
    \item If  $p\in\halfhalf$, and neither $\ybar \leq p$ nor  $\ybar \leq \pbar$,
    then
    $\semantics{p}^* = \semantics{p}$ and
    $\semantics{\pbar}^* = \semantics{\pbar}$.
\end{enumerate}
Again, at most one of every pair of complementary nouns in $\halfhalf$ gets a larger interpretation in $\Nodel$
than in $\Model$.   
This completes the definition of $\Nodel$.
Now this model $\Nodel$ does not change the sizes of any interpretations.  
(This is where we use the assumption that in $\Model$,  $\card{\semantics{y}} \leq
\card{\semantics{\xbar}}$.)
So it satisfies the same
$\existsgeq$ and $\existsmore$ sentences as $\Model$.    In particular, it satisfies all of the 
$\existsgeq$ and $\existsmore$ sentences which happen to belong to $\Gamma$.

\begin{claim}
    $\Nodel$ satisfies all sentences $\forall(a,b)$ which 
    are true in $\Model$, hence all sentences $\forall(a,b)$ which 
    belong to $\Gamma$.
    \label{claim-all-sentences-N}
\end{claim}

\begin{proof}
    We have the following cases:
    (1) $a,b\in\smallsmall$;
    (2) $a\in \smallsmall$, $b\in\halfhalf$;
    (3) $a\in\smallsmall$, $b\in\largelarge$;
    (4) $a,b\in\halfhalf$;
    The other possible cases (such as $a,b\in \largelarge$) follow from these by taking complements.

    (1) Suppose that both $a, b\in\smallsmall$.
    Since our sentence $\forall(a,b)$ is in $\Gamma$, and $\Model\models\Gamma$, $\semantics{a}\subseteq \semantics{b}$.
    Note that $\semantics{a}^*$ is either $\semantics{a}$ or the larger set $\semantics{a}\cup\semantics{x}$.
    If $\semantics{a}^* = \semantics{a}$, then clearly 
    $\semantics{b}^* \supseteq \semantics{a}^*$.
    And if $\semantics{a}^* = \semantics{a}\cup\semantics{x}$, then $\ybar \leq a$.
    But then $\ybar\leq b$ as well, and so 
    $\semantics{b}^* = \semantics{b}\cup\semantics{x}$.
    So again we have $\semantics{a}^* \subseteq \semantics{b}^*$.
    
    (2) Suppose that $a\in \smallsmall$, $b\in\halfhalf$.
    
    The only extra step beyond what we saw in (1) is for the case
    when $\semantics{a}^* = \semantics{a}$
    and $\semantics{b}^* = \semantics{b} \cap \semantics{\xbar}$.
    This case happens when  $\nott(\ybar \leq a)$, but $\ybar \leq \bbar$.
    So $b\leq y$.     We must show that $\semantics{a}\subseteq \semantics{\xbar}$.
    For if not, $\semantics{a}\cap\semantics{x} \neq\emptyset$.
    Then $\Gamma\proves\exists(a,x)$; this is due to the fact that both $a$ and $x$ are in $\smallsmall$. 
    But since $a \leq b \leq y$, we have $\exists(x, y)$ from $\Gamma$.   
    This  contradicts the basic assumption in this section.

    (3) We next consider the case $a\in\smallsmall$, $b\in\largelarge$.
    Let us first assume that $\semantics{a}^* = \semantics{a}$.
    If $\semantics{b}^* = \semantics{b}$, then of course we are done.
    So we assume that $\semantics{b}^* = \semantics{b} \cap\semantics{\xbar}$, and thus that $\ybar\leq \bbar$.
    So $b\leq y$.  
    The rest of the argument at this point is exactly what we saw in (2) just above.

    We continue with (3),  turning to the subcase  $\semantics{a}^* = \semantics{a}\cup\semantics{x}$.
    So $\ybar \leq a$.
    We show that
    $\semantics{b}^* =\semantics{b}$  and that
    $\semantics{x}\subseteq \semantics{b}$.
    Suppose that $\ybar \leq \bbar$.
    As we noted above, $\Gamma\proves\exists(\ybar,\ybar)$.  Thus we have $\exists(a,\bbar)$; this contradicts $a\leq b$.
    This contradiction goes to show that  $\nott (\ybar \leq \bbar)$.
    Thus $\semantics{\bbar}^* = \semantics{\bbar}$.
    It follows that  $\semantics{b}^* = \semantics{b}$.
    To show that $\semantics{a}^*\subseteq \semantics{b}^*$, we only need to see that 
    $\semantics{x}\subseteq \semantics{b}$.
    Now if not, 
    $\semantics{x}\cap \semantics{\bbar} \neq \emptyset$.
    Since $x$ and $\bbar$ belong to $\smallsmall$,   $\Gamma\proves\exists(x,\bbar)$.
    But $\bbar \leq \abar \leq y$.   Thus $\Gamma\proves\exists(x,y)$; this  again
    contradicts the basic assumption in this section.
    
    Finally, (4) is when $a,b\in\halfhalf$.
    If $\semantics{a}^* = \semantics{a} \cup \semantics{x}$,
    then we have $\ybar\leq a$.   So in this case, $\ybar\leq b$, and thus
    $\semantics{b}^* = \semantics{b} \cup \semantics{x}$.
    So $\semantics{a}^* \subseteq \semantics{b}^*$ in this case.
    We thus have the case $\semantics{a}^* = \semantics{a}$ because $\nott(\ybar\leq a)$.
    The only interesting subcase on $\semantics{b}^*$ is when 
    $\semantics{b}^* =  \semantics{b} \cap\semantics{\xbar}$, and thus that $\ybar\leq \bbar$.
    The argument is exactly as in (2) above, with the slight change that $a\in\halfhalf$ rather than $\smallsmall$.
\end{proof}

\begin{claim}
    $\Nodel$ satisfies all $\exists(a,b)$ sentences in $\Gamma$.
    \label{claim-all-exists-N}
\end{claim}

\begin{proof}
    Take a sentence $\phi = \exists(a,b)$ from $\Gamma$.   We know that $\Model\models\exists(a,b)$.
    We still must check that  $\exists(a,b)$ is true in $\Nodel$.
    This is obvious when $a, b\in \smallorhalf$.
    As we know, when $a,b\in\largelarge$, we always have $\exists(a,b)$ on grounds of cardinality.
    We thus need only check the case when $a$ (say) belongs to $\smallorhalf$ and $b$ to $\largelarge$.
    We know from Section~\ref{claim-all-exists} that $\phi$ itself belongs to $\semantics{a}\cap\semantics{b}$.
    We thus only need to show that $\phi\in\semantics{b}^*$ when 
    $\semantics{b}^* = \semantics{b} \cap \semantics{\xbar}$; i.e., when
    $\ybar \leq \bbar$. 
    So we must show that $\phi\notin \semantics{x}$. For if we had $\exists(a,b) \in\semantics{x}$, then either 
    $a\leq x$ or $b\leq x$.
    In the case that $a\leq x$, we have $a \leq x$ and $b\leq y$.  So since  $\exists(a,b)\in \Gamma$,
    $\Gamma\proves\exists(x,y)$.
    This also happens in the second case, when 
    $b\leq x$.   For in this case, we have $b\leq x, y$.  Together with $\exists(a,b)\in \Gamma$, we  again see that $\Gamma\proves\exists(x,y)$.
    Either way, we contradict the basic assumption in this section.
\end{proof}

We have proven Claim~\ref{claim-all-exists-N}.   As a result, $\Nodel\models \Gamma$.
But $\semantics{\ybar}^*  = \semantics{x} \cup \semantics{\ybar} = \semantics{x}^* \cup \semantics{\ybar}$.
Thus $\semantics{x}^* \subseteq \semantics{\ybar}^*$.
It follows that    $\semantics{x}^* \cap \semantics{y}^*  = \emptyset$.
The upshot is that $\Nodel\not\models \exists(x,y)$, as desired. 

This concludes the work in Section~\ref{section-exists-full-result}.
We have proved Theorem~\ref{completeness}.

\begin{corollary}
    If $\Gamma\not\proves\phi$, then $\Gamma\cup\set{\phibar}$ is consistent.
\end{corollary}

As a corollary to the proof of our theorem, we also have the following result.

\begin{theorem}
    If $\Gamma$ has a model, then $\Gamma$ has a model of size $\aleph_n$ for some $n\in N$.
\end{theorem}

Since the proof system is direct (i.e., \emph{reductio ad absurdum} is not used), we also have
the following result.

\begin{theorem}
    The question of whether $\Gamma\models\phi$ or not can be decided in logspace.
\end{theorem}

See~\cite{phmoss} for a general discussion of complexity matters concerning syllogistic logics.
 The reason for the efficient proof search is that the logic did not use \emph{reductio ad absurdum}.
 Instead it uses  \emph{ex falso quo 
quodlibet}.  In effect, one tries to see whether $\Gamma\proves\phi$ by generating all of the proofs from $\Gamma$. 
One needs to know that if $\Gamma\proves \phi$, then a derivation may be found using only the atoms in $\Gamma\cup\set{\phi}$ and their complements.
We omit the details.
 There are only
polynomially many such sentences, and the proof search is a generalization of searching for paths in graphs.
If $\Gamma\not\proves\phi$, then there is a counter-model.   It may be chosen to be of size $\aleph_n$ for some natural number $n$.
Further, it may be taken to be finitely describable in a strong sense.

\begin{remark}
One way to extend the completeness result in this paper is to add the boolean connectives $\andd$ and $\nott$ to the sentences.  
So in effect, one has propositional logic with the sentences of $\Sdagleq$
as atomic sentences.   The semantics is the obvious one.   For the proof theory, one takes a Hilbert style axiomatization of propositional logic, 
and adds implicational sentences corresponding to the rules of $\Sdagleq$.  (For example, corresponding to ({\sc non-empty}),
we would have $\existsgeq(p,\pbar) \iif \existsgeq(\pbar,x)$.)
The argument is an adaptation of one from~\cite{logic:moss08}.
The completeness boils down to showing that every  consistent sentence $\phi$ in the new logic has a model. 
Using disjunctive normal forms, we may assume that $\phi$ is a conjunction of atomic sentences or their negations.
But the atomic sentences in this language are closed under (semantic) negation, so we may assume that $\phi$ is a conjunction  $\bigwedge S$ of 
a (finite) set $S$ of sentences of $\Sdagleq$.   And here, we claim that the set of conjuncts of $\phi$ is consistent in $\Sdagleq$. 
For this, we argue by induction on proofs in $\Sdagleq$ 
which don't use the ({\sc x}) rules
that if $S\proves\psi$ in $\Sdagleq$, then $\proves \bigwedge S \iif \psi$.
\end{remark}

\section{Examples}
\label{section-examples}

Let $\Gamma$ be the set of sentences shown in Figure~\ref{figure-example-Gamma}.

\begin{ex}
    Notice that in any model of $\Gamma$, $\semantics{c}$ and $\semantics{d}$ would be sets of the same size, 
    and yet $\semantics{c}$ is a
    proper subset of $\semantics{d}$.    So $\semantics{c}$ and $\semantics{d}$
    must be infinite.
    Further, since we have $\existsmore(e,c)$ in $\Gamma$, $\semantics{e}$ must be uncountable.
\end{ex}

\begin{ex}
    The standard partition of $\Gamma$ is
    $$\begin{array}{lcl}
    \smallsmall & = & \set{a,b,c,d} \\
    \halfhalf & = & \set{e,\ebar,f,\fbar, g, \gbar} \\
    \largelarge & = & \set{\abar,\bbar,\cbar,\dbar}
    \end{array}
    $$
\end{ex}

\begin{figure}[p]
    \begin{mathframe}
        $$\begin{array}{l@{\qquad}l@{\qquad}l@{\qquad}l}
        \forall(a,\abar)  & \existsmore(c,b) 
        & \existsgeq(c,d)  &
        \existsgeq(d,c) \\
        \existsmore(e,c) & 
        \existsgeq(e,\ebar) & \existsgeq(\ebar,e) 
        & \existsgeq(f, e) \\
        \existsgeq(e, f)  &
        \existsgeq(g, b) & \forall(c,d) & \exists(d, \cbar)  \\
        \exists(c,f) & \forall(e,f)  &
        \exists(\ebar,f)  &
        \exists(\ebar,\fbar) \\
        \exists(c,e) & 
        \end{array}
        $$
        
        \medskip
        The relation $\leq$ is given by  $a\leq x$ for all $x$; 
        and  $c \leq d$, 
        and
        $e\leq f$.
        We also have the 
        relations derived from these by ({\sc anti}):  
        $x\leq \abar$ for all $x$;
        $\dbar \leq \cbar$,  
        and
        $\fbar \leq \ebar$.
        
        \medskip
        We get $         \phi_1=\exists(d, \cbar),\: \phi_2=\exists(\ebar,f),\: \phi_3=\exists(\ebar,\fbar),\: \phi_4=\exists(c,e)$ and  $ \phi_5=\exists(c,f) $ for $ \Gamma_\exists $ sentences.
        
        \medskip
        The relation $\leq_c$ has all of the pairs in $\leq$ above, and also 
        $b \leq_c  c$, $c \leq_c d \leq_c c$, $c\leq_c e$, $e \leq_c \ebar \leq_c e$,  $e \leq_c f \leq_c e$, 
        $f \leq_c \fbar \leq_c f$, and 
        $b \leq_c g$; also, for the all nouns $n$, 
        $n \leq_c \abar,\bbar,\cbar, \dbar, e, \ebar, f, \fbar$.
        
        \medskip  
        The strict part $<_c$ is 
        $a <_c x $ for all $x$ other than $a$;
        also $b<_c c$, $c <_c e, \ebar, f, \fbar$, and $b <_c g$.
        
        \medskip  
        The relation $<_{more}$ is given by
        $a, b \leqmore x $ for all $x$ other than $a$, $b$, $g$, $\gbar$; 
        also $a, b, c, d \leqmore y$ for all $y$ other than $a$, $b$, $c$, and $d$;
        and  $a,b, c,d \leqmore \abar, \bbar, \cbar, \dbar, e, \ebar, f, \fbar$.
        
        \medskip
        The relation $\equiv$ is the identity and also
        $c\equiv_c d$, and 
        $e \equiv_c \ebar \equiv_c f \equiv_c \fbar$.

        \medskip
        
        A model:
        \begin{small}
            $$\begin{array}{ll}
            \begin{array}{lcl}
            \semantics{a} & = & \emptyset \\
            \semantics{b} & = & \emptyset \\
            \semantics{c} & = & A_1 + \set{\phi_4, \phi_5}  \\
            \semantics{d} & = &(A_1  \cup  A_2 )+ \set{\phi_1,\phi_4,\phi_5}  \\
            \semantics{e} & = &  B_1 + \set{\phi_4} \\
            \semantics{f} & = &   (B_1 \cup B_2) + \set{\phi_2,\phi_4,\phi_5} \\
            \semantics{g} & = &  B_3 + \set{\phi_1,\phi_2,\phi_3,\phi_4,\phi_5} \\
            \end{array}
            &
            \begin{array}{lcl}
            \semantics{\abar} & = & (A_1 \cup A_2) + (B_1\cup B_2\cup B_3) + \set{\phi_1, \ldots, \phi_5}
            \\
            \semantics{\bbar} & = &(A_1 \cup A_2) + (B_1\cup B_2\cup B_3) + \set{\phi_1, \ldots, \phi_5}
            \\
            \semantics{\cbar} & = & A_2 +  (B_1 \cup B_2\cup B_3) + \set{\phi_1,\phi_2,\phi_3} 
            \\
            \semantics{\dbar} & = &     (B_1 \cup B_2\cup B_3) + \set{\phi_2,\phi_3} 
            \\
            \semantics{\ebar} & = & (A_1 \cup A_2) + (B_2 \cup B_3) + \set{\phi_1,\phi_2, \phi_3,\phi_5} \\
            \semantics{\fbar} & = & (A_1 \cup A_2) +  B_3 + \set{\phi_1,\phi_3} \\
            \semantics{\gbar} & = &(A_1 \cup A_2) + (B_1\cup B_2)\\
            \end{array}
            \end{array}
            $$
        \end{small}

        Here $A_1$ and $A_2$ are disjoint countable sets; $B_1$, $B_2$, $B_3$ are pairwise disjoint sets of size 
        $\aleph_1$; and $\phi_1$, $\ldots$, $\phi_5$ are the $\exists$ sentences in $\Gamma$ as in (\ref{numbered-phis}).

    \end{mathframe}
    \caption{At the top is shown a  set  of sentences $\Gamma$ which is  used 
        as a running example throughout Section~\ref{section-examples}.
        Below is some information about the relations derived from $\Gamma$.
        At the bottom is one example of a model which may be found by the method of Section~\ref{section-model-construction}.
        \label{figure-example-Gamma}}
\end{figure}

\begin{ex}
    The set $Q$ in (\ref{eq-Q}) is $\set{c,d}$.    This set is an equivalence class for $\equiv_c$.
    The one and only listing of $Q$ is the one-term sequence $\set{c,d}$.
    
    In terms of our earlier notation,  $[p_1] = \set{c, d} $;  for later, $p_2 = \halfhalf$.

    $\Delta_1$ is the set 
    $$\set{ \exists(c,c), \exists(d,d),  \forall(c,d),  \exists(d, \cbar)}
    $$
    One possible finite model of $\Delta_1$  with the right properties has universe $\set{1,2}$, $\semantics{c} = \set{1}$, and $\semantics{d} = \set{1,2}$.
    So for our $\Structure_1$ we take two disjoint sets $A_1$ and $A_2$, each of cardinality $\aleph_0$,  and then set 
    $\Structure_1 = A_1 \cup A_2 $, 
    $\semantics{c} = A_1$, and $\semantics{d}  = A_1 \cup A_2$.  
    Next, $n+1 = 2$ demands a larger size.   So we set $\kappa_2 = \aleph_1$.
    Also,
    $$\begin{array}{cccc}
    \Delta_2  & = & \set{ \exists(e,e), \exists(\ebar,\ebar), \exists(f,f), \exists(\fbar,\fbar), \exists(g,g), \exists(\gbar,\gbar),
        \forall(e,f),
        \exists(\ebar,f), \exists(\ebar,\fbar)}
    \end{array}
    $$
    One finite structure of $\Delta_2$  has universe $\set{1,2,3}$,
    $\semantics{e} = \set{1}$, $\semantics{f} = \set{1,2}$,  and $\semantics{g}  = \set{3}$.
    (Note that $\semantics{g}$ could also be $\set{1}$ or $\set{2}$.)
    So for $\Structure_2$, we take disjoint sets $B_1$, $B_2$, $B_3$ of size $\aleph_1$, and 
    $S_2 = B_1 \cup B_2\cup B_3$, 
    $\semantics{e}  = B_1$, $\semantics{f} = B_1\cup B_2$, and $\semantics{g} = B_3$.
    Thus 
    $\semantics{\ebar}  = B_2\cup B_3$, $\semantics{\fbar} = B_3$, and $\semantics{\gbar} = B_1\cup B_2$.
    
    We are almost ready to define the model of $\Gamma$ produced by our method.
    We take the $\exists$ sentences in $\Gamma$ and number them:
    \begin{equation}
    \begin{array}{lllll} \phi_1 = 
    \exists(d, \cbar)    &
    \phi_2 =\exists(\ebar,f)  &  \phi_3 =\exists(\ebar,\fbar) & 
    \phi_4 = \exists(c,e) & \phi_5 = \exists(c,f)
    \end{array}
    \label{numbered-phis}
    \end{equation}
    And now we can say what our structure $\Structure$ is.  The universe is
    $$(A_1  \cup A_2) + (B_1 \cup B_2 \cup B_3) + \set{\phi_1, \phi_2, \phi_3, \phi_4, \phi_5}
    $$
    where the $A_i$ have size $\aleph_0$ and the $B_j$ have size $\aleph_1$.  The interpretations of the nouns is
    $$\begin{array}{ll}
    \begin{array}{lcl}
    \semantics{a} & = & \emptyset \\
    \semantics{b} & = & \emptyset \\
    \semantics{c} & = & A_1 + \set{\phi_4, \phi_5}  \\
    \semantics{d} & = &(A_1  \cup  A_2 )+ \set{\phi_1,\phi_4,\phi_5} \\
    \semantics{e} & = &  B_1 + \set{\phi_4} \\
    \semantics{g} & = &  B_3 + \set{\phi_1,\phi_2,\phi_3,\phi_4,\phi_5}\\
    \end{array}
    &
    \begin{array}{lcl}
    \\
    \\
    \\
    \\
    \semantics{\ebar} & = & (A_1 \cup A_2) + (B_2 \cup B_3) + \set{\phi_1,\phi_2, \phi_3} \\
    \semantics{\fbar} & = & (A_1 \cup A_2) +  B_3 + \set{\phi_1,\phi_3} \\
    \semantics{\gbar} & = & (B_1\cup B_2)\\
    \end{array}
    \end{array}
    $$

\end{ex}

\begin{ex}  The work on refinements in Section~\ref{section-refine-three-classes}
    applies to both $g$ and also to $\gbar$.   So we would get two different partitions with all the required properties:
    $$
    \begin{array}{l | l}
    \begin{array}{lcl}
    \smallsmall & = & \set{a,b,c,d,g} \\
    \halfhalf & = & \set{e,\ebar,f,\fbar} \\
    \largelarge & = & \set{\abar,\bbar,\cbar,\dbar,\gbar}
    \end{array}
    \qquad & \qquad 
    \begin{array}{lcl}
    \smallsmall & = & \set{a,b,c,d,\gbar} \\
    \halfhalf & = & \set{e,\ebar,f,\fbar} \\
    \largelarge & = & \set{\abar,\bbar,\cbar,\dbar,g}
    \end{array}
    \end{array}
    $$
\end{ex}

\section{Conclusion}

This paper has shown a soundness and completeness theorem of the logic $\Sdagleq$ 
when interpreted
on infinite sets. 
The main point of the paper is that 
the language has cardinality comparison features are not expressible in first order logic,
and yet  it has a complete proof system
with  
  an efficient proof search  algorithm.     Some may find these results to be surprising.
  The message is that the very weak system of syllogistic logic has extensions which are well-behaved and expressive.
  
  This paper contributes to the project of finding larger and larger well-behaved logics which talk about sizes of sets.
 One next question would be to take the logic of~\cite{moss:forDunn} and add the quantifier {\sf most}, interpreted by strict majority.
  
 \rem{
We would add boolean operations. In this case, every sentence of the logic is a propositional atom of the new logic. An important point would be  $ \exists(x,y) \wedge \exists(y,x)  $  for any variables $ x,y $  since $AC $.

We hope that the construction techniques in this paper will help with other results in the area.
}
\bibliography{natlogic}

\vspace{3cm}

Lawrence S.~Moss, Indiana University, Department of Mathematics, Bloomington, USA.
\thanks{This work was partially supported by a grant from the Simons Foundation ($\#$245591 to Lawrence Moss)
\indent E-mail: lmoss@indiana.edu\vspace{1cm}

Sel\c{c}uk Topal,  Bitlis Eren University, Department of Mathematics, Bitlis, Turkey.

\indent E-mail: s.topal@beu.edu.tr

\end{document}